\documentclass[12pt,reqno]{amsart}
\usepackage{amsmath, amsfonts, amssymb}
\usepackage{latexsym}
\usepackage{enumerate}
\usepackage{cite}
\usepackage{bm}
\usepackage{mathtools}
\usepackage{eucal}
\usepackage{mathrsfs}
\usepackage[usenames]{color}
\usepackage{xcolor}
\definecolor{refkey}{gray}{.45}
\definecolor{labelkey}{gray}{.45}

\numberwithin{equation}{section}

\newcommand{\rn}{\mathbb R^n}
\newcommand{\sn}{S^{n-1}}

\newcommand{\blb}{\raise.3ex\hbox{$\scriptstyle \pmb \lbrack$}}
\newcommand{\sblb}{\raise.1ex\hbox{$\scriptscriptstyle \pmb \lbrack$}}
\newcommand{\brb}{\raise.3ex\hbox{$\scriptstyle \pmb \rbrack$}}
\newcommand{\sbrb}{\raise.1ex\hbox{$\scriptscriptstyle \pmb \rbrack$}}
\newcommand{\bla}{\raise.2ex\hbox{$\scriptstyle\pmb \langle$}}
\newcommand{\sbla}{\raise.1ex\hbox{$\scriptscriptstyle\pmb \langle$}}
\newcommand{\bra}{\raise.2ex\hbox{$\scriptstyle\pmb \rangle$}}
\newcommand{\sbra}{\raise.1ex\hbox{$\scriptscriptstyle\pmb \rangle$}}
\newcommand{\blrb}{\raise.3ex\hbox{$\scriptstyle \pmb | $}}
\newcommand{\sblrb}{\raise.1ex\hbox{$\scriptscriptstyle \pmb | $}}

\newcommand{\wt}{\widetilde}
\newcommand{\psum}{\,{+_{\negthinspace\kern-2pt p}}\,}
\newcommand{\qsum}[1]{\,{+_{\negthinspace\kern-2pt \lower -2pt \hbox{$_{_{#1}}$}}}\,}
\newcommand{\osum}{{+_{\negthinspace\kern-2pt {\rm{o}}}}\,}
\newcommand{\dpsum}{\,{\tilde+_{\negthinspace\kern-1pt p}}\,}
\newcommand{\dqsum}[1]{{\,\wt+_{\negthinspace\kern-1pt #1}}\,}
\newcommand{\lsub}[1]{\hskip -1.5pt\lower.5ex\hbox{$_{#1}$}}

\newcommand{\R}{\mathbb{R}}

\makeatletter
\newcommand{\LeftEqNo}{\let\veqno\@@leqno}
\makeatother
\DeclareMathOperator{\Span}{span}

\newtheorem{lemma}{Lemma}[section]
\newtheorem{theorem}[lemma]{Theorem}

\author[K. B\"or\"oczky]{K\'aroly J. B\"or\"oczky}
\address[K. B\"or\"oczky]{Alfr\'ed R\'enyi Institute of Mathematics
 Hungarian Academy of Sciences and Central European University}
\author[E. Lutwak]{Erwin Lutwak}
\address[E. Lutwak, D. Yang, G. Zhang, Y. Zhao]{Department of Mathematics, New York University Tandon School of Engineering}
\author[D. Yang]{Deane Yang}
\author[G. Zhang]{Gaoyong Zhang}
\author[Y. Zhao]{Yiming Zhao}
\title{The dual Minkowski problem for symmetric convex bodies}
\begin{document}

\maketitle

\begin{abstract}
  The dual Minkowski problem for even data asks what are the necessary and sufficient conditions on an even prescribed measure on the unit sphere for it to be the $q$-th dual curvature measure of an origin-symmetric convex body in $\R^n$. A full solution to this is given when $1 < q < n$. 
  The necessary and sufficient condition is an explicit measure concentration condition. A variational approach is used, where the functional is the sum of a dual quermassintegral and an entropy integral. The proof requires two crucial estimates. The first is an estimate of the entropy integral proved using a spherical partition. The second is a sharp estimate of the dual quermassintegrals for a carefully chosen barrier convex body.
\end{abstract}

\section{Introduction}

Geometric measures and their associated Minkowski problems of convex
bodies in Euclidean space are central to the study of convex geometry.
The classical Minkowski problem is
prescribing the surface area measure (in the smooth case, prescribing the Gauss curvature) of a convex body. The solution to
the classical Minkowski problem
has had many applications in
various fields of analysis and geometry.
See Section 8.2 in Schneider \cite{schneider2014} for an overview.
The Christoffel-Minkowski problem (prescribing area measures)
and the Aleksandrov problem (prescribing curvature measures)
are two other important Minkowski problems in convex geometry
 that are still unsolved.
See for example, Sections 8.4 and 8.5 in \cite{schneider2014}.
These Minkowski problems belong to the classical Brunn-Minkowski theory.

More recently, Lutwak \cite{MR1231704}
introduced
the $L_p$ Brunn-Minkowski theory, where $p=1$ is the classical theory cited above, and posed the  $L_p$ Minkowski problem (prescribing $L_p$ surface area measure) as a fundamental question.
The most important (and therefore most challenging) cases include, when $p=0$, the logarithmic Minkowski
problem (see B\"{o}r\"{o}czky-Lutwak-Yang-Zhang \cite{BLYZ})
 and, when $p = -n$, the centro-affine Minkowski problem (see
 Chou \& Wang \cite{MR2254308} and Zhu \cite{MR3356071}).
 The $L_p$ Minkowski problem when $p>1$ was solved by Lutwak
 \cite{MR1231704} for symmetric convex bodies and by Chou \& Wang \cite{MR2254308}
 for the general case. Alternate proofs were given by Hug-Lutwak-Yang-Zhang \cite{HLYZ05}. The case of $p<1$ is still largely open (see
 B\"{o}r\"{o}czky-Lutwak-Yang-Zhang \cite{BLYZ}, Huang-Liu-Xu \cite{HLX}, Jian-Lu-Wang \cite{JLW},
and Zhu \cite{MR3228445, Zhu-jfa}). For other recent progress on the $L_p$-Minkowski problem, see B\"{o}r\"{o}czky-Trinh \cite{BT} and Chen-Li-Zhu \cite{CLZ1, CLZ2}.
 The $L_p$ Minkowski problem also plays a key role
 in establishing affine Sobolev inequalities (see, for example,
 Lutwak-Yang-Zhang \cite{LYZ1, LYZ2}, Cianchi-Lutwak-Yang-Zhang \cite{CLYZ}, and Haberl \& Schuster \cite{MR2530600}).

 Very recently,
 Huang-Lutwak-Yang-Zhang \cite{HLYZ}
introduced  \emph{dual curvature measures}
$\widetilde{C}_q$, where $q\in \mathbb{R}$,
as the natural duals to Federer's curvature measures. These are fundamental to the dual Brunn-Minkowski theory and analogous to the surface area measures in Brunn-Minkowski theory mentioned above.
This leads naturally to the \emph{dual Minkowski problem} of prescribing
dual curvature measures.
Remarkably, the family of dual Minkowski problems connects the well-known
Aleksandrov problem ($q=0$) to the logarithmic Minkowski problem ($q=n$) mentioned above.
We present here a complete solution to
the dual Minkowski problem  within the class of origin-symmetric convex bodies
for the critical strip $0<q<n$.

The dual Brunn-Minkowski theory was first introduced by Lutwak, based on
a conceptual but mysterious duality\footnote{Although Lutwak's duality is motivated by the duality between intersections and projections in projective geometry, it is a duality of concepts (such as mixed volumes) instead of the usual duality between points and hyperplanes in a vector space.}  in convex geometry (see Schneider \cite{schneider2014},
p. 507, for a lucid explanation). The power of the theory
was demonstrated when
intersection bodies, which are central to the dual Brunn-Minkowski theory,
played a crucial role in the solution to the well-known Busemann-Petty problem. The proof relied on connections between the dual theory and harmonic analysis.
See, for example, Bourgain \cite{Bo91}, Gardner \cite{G94annals},
Gardner-Koldobsky-Schlumprecht \cite{GKS99}, Lutwak \cite{L88Adv},
and Zhang \cite{Z99annals},
and see Gardner \cite{MR2251886} and Koldobsky \cite{K05}
for additional references.

Dual curvature measures, parameterized by $q \in \R$,
are the analogues in the dual Brunn-Minkowski theory to
Aleksandrov's area measures and Federer's curvature measures
in the classical Brunn-Minkowski theory.
The $0$-th dual curvature measure
is equivalent to Aleksandrov's integral curvature for
the polar body. The $n$-th dual curvature measure is
the cone volume measure studied in Barthe,
Gu\'{e}don, Mendelson \& Naor \cite{MR2123199}, B\"{o}r\"{o}czky \& Henk
\cite{MR3415694}, Henk \& Linke \cite{MR3148545}, Ludwig \& Reitzner
\cite{MR2680490}, Stancu \cite{MR1901250,MR2019226}, and Zou \& Xiong
\cite{MR3255458}. Dual curvature measures encode the geometry of a convex body's interior, in contrast to their counterparts in the Brunn-Minkowski theory, which reflect the geometry of the boundary.
They provide a new class of valuations (i.e., finitely additive geometric invariants of convex bodies) that are dual to their counterparts in the Brunn-Minkowski theory. The latter have been studied extensively in recent years. See
Haberl \cite{MR2966660},
Haberl \& Ludwig \cite{MR2250020}, Haberl \& Parapatits \cite{MR3194492,MR3176613},
 Ludwig \cite{MR2159706,MR2772547}, Ludwig \& Reitzner \cite{MR2680490},
 Schuster \cite{MR2435426,MR2668553}, Schuster \& Wannerer \cite{MR2846354},
 Zhao \cite{zhao15imrn}
 and the references therein.
\smallskip

The
dual Minkowski problem
for dual curvature measures proposed in Huang-Lutwak-Yang-Zhang \cite{HLYZ}
states:
\smallskip

\noindent
\textbf{The Dual Minkowski Problem.} \textit{
Given a finite Borel measure $\mu$ on the unit sphere $S^{n-1}$
and a real number $q\in \mathbb{R}$,
find necessary and sufficient conditions on $\mu$ so that there exists a
convex body $K$ in $\rn$ that solves the geometric equation,
\begin{equation}\label{dMPeq}
\widetilde{C}_q(K,\cdot)=\mu,
\end{equation}}
where $\widetilde{C}_q(K,\cdot)$ is the $q$-th dual curvature measure of the convex body
$K$ in $\rn$.

In the special case when the given measure has a density $f$,
the geometric equation \eqref{dMPeq}
reduces to the Monge-Amp\`ere type equation on $S^{n-1}$ given by
\begin{equation}
  \label{pde1.1}
\operatorname{det}(\overline\nabla^2h + hI) = nh^{-1}(|\overline\nabla h|^2 + h^2)^{(n-q)/2}f
\end{equation}
where $f$ is a given non-negative integrable function, $h$ is the unknown function, $I$ is the standard Riemannian metric on $S^{n-1}$. Also, $\overline\nabla h$ and $\overline\nabla^2h$ are the gradient and Hessian of $h$ with respect to $I$, respectively.

Dual Minkowski problems, including the logarithmic Minkowski problem, are more
challenging than previously solved Minkowski problems.
This arises from
 the phenomenon of  measure concentration, which implies
 that there are singular prescribed measures for which
no solutions exist.
 This implies that there is no straightforward way to solve general problem by first solving the smooth case using the partial differential equation
\eqref{pde1.1} and then using an  approximation argument.

When $q=0$, the dual Minkowski problem is the classical Aleksandrov problem
that was solved by Aleksandrov \cite{MR0007625} using a topological argument. See also
Guan \& Li \cite{MR1454174}, Oliker \cite{MR2332603}, and Wang
\cite{MR1357407} for other work on this problem and its variants.
The $L_p$ version of the Aleksandrov problem was introduced and studied
by Huang-Lutwak-Yang-Zhang \cite{HLYZ2}.

When $q = n$, the dual Minkowski problem is the logarithmic Minkowski problem
that was solved by B\"or\"oczky-Lutwak-Yang-Zhang \cite{BLYZ} for symmetric convex bodies and is still
open in the asymmetric case (see, for example, B\"{o}r\"{o}czky, Heged\H{u}s \& Zhu
\cite{Boroczky20062015}, Stancu \cite{MR1901250,MR2019226},
 Zhu \cite{MR3228445}). The logarithmic Minkowski problem is
 closely connected with isotropic measures
 (B{\"o}r{\"o}czky-Lutwak-Yang-Zhang \cite{MR3316972}) and curvature flows
 (Andrews \cite{MR1714339, MR1949167}). It was discovered that a measure concentration condition (described in the next paragraph) is the precise obstruction to the existence of solutions to this singular Monge-Amp\`ere equation.

   A finite Borel measure
   $\mu$ on $S^{n-1}$ is said to satisfy the \emph{subspace concentration
    condition} if
\begin{equation}
\label{scc}
\frac{\mu(\xi\cap S^{n-1})}{\mu(S^{n-1})}\leq \frac{\dim\xi}{n},
\end{equation}
for each proper subspace $\xi\subset \mathbb{R}^n$ and, if
equality holds for a subspace $\xi$,
there exists a subspace $\xi'\subset \mathbb{R}^n$ complementary
to $\xi$ such that $\mu$ is concentrated on $S^{n-1}\cap (\xi\cup \xi')$.
B\"{o}r\"{o}czky-Lutwak-Yang-Zhang \cite{BLYZ} proved that there exists an origin-symmetric
convex body whose cone volume measure is equal to $\mu$ if and only if $\mu$
 is an even finite Borel measure that satisfies the subspace concentration
 condition.

The same phenomenon arose in Huang-Lutwak-Yang-Zhang's attempt \cite{HLYZ} to solve the dual Minkowski problem for symmetric convex bodies .
They defined, for an even finite Borel measure
$\mu$ on $\sn$ and $1<q<n$, the following subspace mass inequalities:
\begin{equation}
\label{smi1}
\frac{\mu(\xi\cap S^{n-1})}{\mu(S^{n-1})}<1- \frac{q-1}q \frac{n-\dim\xi}{n-1},
\end{equation}
for each proper subspace $\xi\subset \mathbb{R}^n$,
and, for $0<q\le 1$,
\begin{equation}
\label{smi1.1}
\frac{\mu(\xi\cap S^{n-1})}{\mu(S^{n-1})}<1,
\end{equation}
for any subspace $\xi$ of co-dimension 1. They proved that, if satisfied,
there exists an
origin-symmetric convex body whose $q$-th dual curvature measure is equal
to $\mu$. That is, the subspace mass inequalities \eqref{smi1} and \eqref{smi1.1}
are sufficient conditions for existence of solutions to
the dual Minkowski problem for symmetric convex bodies.
When $0<q\le 1$,
condition \eqref{smi1.1} is both necessary and sufficient, but when $1<q<n$, condition \eqref{smi1}
 is not a necessary one. In fact, examples of origin-symmetric convex bodies whose dual
 curvature measures violate \eqref{smi1} were presented
  recently in \cite{BH,zhao}.
  A more refined subspace mass inequality, which first appeared in \cite{BH,zhao},
  is the following:
\smallskip

\noindent
{\bf Subspace Mass Inequality.}\
For $0<q<n$, a finite Borel measure $\mu$ on $S^{n-1}$ is said to satisfy
\emph{the $q$-th subspace mass inequality} if
\begin{equation}
\label{smi2}
\frac{\mu(\xi\cap S^{n-1})}{\mu(S^{n-1})} <
\begin{cases}
i/q  &i < q \\
1 &i \ge q
\end{cases}
\end{equation}
for any proper subspace $\xi$ of dimension $i$ in $\rn$.
B\"{o}r\"{o}czky,
Henk \& Pollehn \cite{BH} showed that, when $1<q<n$,
the $q$-th subspace mass inequality is a necessary condition for
the existence of solutions to the
dual Minkowski problem for symmetric convex bodies. That is,
the $q$-th
dual curvature measure of every origin-symmetric convex body
satisfies the $q$-th subspace mass inequality. Zhao \cite{zhao} showed
when $q\in \{2, \ldots, n-1\}$,
the $q$-th subspace mass inequality is also a sufficient condition for the
dual Minkowski problem for symmetric convex bodies. That is,
every even finite Borel
measure satisfying the $q$-th subspace mass inequality is the $q$-th
dual curvature measure of an origin-symmetric convex body.
This provides a complete solution to the dual Minkowski problem for even data and integer $q\in \{2, \ldots, n-1\}$.
\smallskip

The aim of this paper is to give a complete solution to the dual Minkowski problem for even data and any real $q\in (0,n)$.

\begin{theorem}
\label{main theorem}
Let $0<q<n$ and $\mu$ be a non-zero even finite Borel measure on $S^{n-1}$.
Then there exists an origin-symmetric convex body $K$ in $\rn$ such that
$\widetilde{C}_q(K,\cdot)=\mu$ if and only if $\mu$ satisfies
the $q$-th subspace mass inequality \eqref{smi2}.
\end{theorem}
When $0<q\leq 1$, the $q$-th subspace mass inequality says nothing
more than that the measure $\mu$ can not concentrate entirely in any
great hypersphere. In this case, Theorem \ref{main theorem} was proved in
\cite{HLYZ}.
When $1<q<n$, the necessity of the $q$-th subspace mass inequality was proved in
\cite{BH} and, its sufficiency,
when $q\in \{2, \ldots, n-1\}$, was proved in \cite{zhao}.

The dual Minkowski problem for $q<0$, as for the classical
Minkowski problem, does not require any non-trivial
measure concentration condition and was solved by Zhao \cite{zhao2}.
The dual Minkowski problem for even data and $q=0$ is equivalent to
the Aleksandrov problem for even data, which was solved by Aleksandrov himself. Another proof appears in \cite{HLYZ2}.
When $q = n$, the dual Minkowski problem for even data is the logarithmic Minkowski problem for even data, which was solved in \cite{BLYZ}.

Unlike the classical Minkowski problem, it is difficult to see how to reduce the $q > 0$ case of the dual Minkowski problem to the case where the measure has a density. Moreover, estimates for the dual quermassintegrals of degree $q > 0$ are much more difficult to obtain than when $q = n$, where the dual quermassintegral is just volume and only an entropy estimate is needed.
When $q > 0$, more delicate estimates for both entropy and the dual quermassintegrals are needed.

The proof presented here uses the variational approach.
The maximization problem associated with the
dual Minkowski problem is described in Section \ref{section maximization problem}. Its solution requires two crucial estimates.
In Section \ref{est1}, we prove an estimate for an entropy integral using the technique of spherical partitions introduced in \cite{BLYZ}. In Section \ref{5.1}, we establish a bound on the dual quermassintegral of a barrier convex body using general spherical coordinates.

The role of a barrier convex body for an integral estimate is the same as
that of a barrier function for PDE estimates.
 Choosing the right barrier and proving a sharp
 estimate are critical to showing that the $q$-th subspace mass inequality is both necessary and sufficient for solving the dual Minkowski problem.
However, for a dual quermassintegrals of any real degree $q > 0$,
the choice of the right barrier is much more subtle than in \cite{HLYZ} and \cite{zhao}. The sharp estimate of its dual quermassintegral then requires a more elaborate set of general spherical coordinates than in \cite{HLYZ}.

 In \cite{HLYZ}, a cross-polytope was used as the barrier to show that condition \eqref{smi1} is sufficient in the cases considered.
 In \cite{zhao}, using the Cartesian product of an ellipsoid and a ball
as the barrier shows that \eqref{smi2} is both
necessary and sufficient, but only for integer $q\in \{2,\ldots, n-1\}$.
 Here, the necessity and sufficiency of \eqref{smi2} for all $q \in (0,n)$ is established by setting the barrier equal to the Cartesian
 product of an ellipsoid, a line segment, and a ball.
 The estimates of its dual quermassintegrals appear in Section \ref{5.1}.

 The work presented here extends significantly the results and techniques
 in \cite{BLYZ}, \cite{HLYZ} and \cite{zhao}.

\section{Preliminaries}
\label{section preliminary}
Basics in the theory of convex bodies will be covered in this section.
More details can be found in the books \cite{MR2251886} and \cite{schneider2014}.

We will work in $\mathbb{R}^n$ equipped with the standard Euclidean norm.
For $x,y\in \mathbb{R}^n$, we write $x\cdot y$ for the inner product
of $x$ and $y$, and let $|x|=\sqrt{x\cdot x}$.
 The unit ball is written as $B^n$ and the unit sphere
as $S^{n-1}$. We use $\omega_n$ for the volume of $B^n$. Recall that
the surface area of $S^{n-1}$ is $n\omega_n$. We will use $C(S^{n-1})$
for the normed vector space of continuous functions on the unit
sphere $S^{n-1}$ equipped with the max norm; i.e.,
$\|f\| = \max\{|f(u)|:u\in S^{n-1}\}$ for each $f\in C(S^{n-1})$.
 Let $C^+(S^{n-1})\subset C(S^{n-1})$ denote the cone of positive
functions, $C_e(S^{n-1})\subset C(S^{n-1})$ the subspace of
even functions, and $C_e^+(S^{n-1}) = C^+(S^{n-1})\cap C_e(S^{n-1})$.
The total measure of a given
finite Borel measure $\mu$ will be written as $|\mu|$. Throughout the paper,
an expression $c(\cdots)$ denotes a ``constant'' whose exact value depends on the parameters listed but may change from line to line. For example, $c(n,k,q)$ is a constant that depends only on $n$, $k$, $q$ and nothing else.
Denote by $\lfloor q \rfloor$ the floor function whose value is the largest
integer less than or equal to $q$.

We say that $K \subset \R^n$ is a \emph{convex body}
if it is a compact convex set with non-empty interior. The boundary
of $K$ is written as $\partial K$. The set of all convex bodies is
denoted by $\mathcal{K}^n$. The set of all convex bodies containing
the origin in the interior is denoted by $\mathcal{K}_o^n$, and the
set of all origin-symmetric convex bodies by
$\mathcal{K}_e^n$. Obviously,
$\mathcal{K}_e^n\subset \mathcal{K}_o^n\subset \mathcal{K}^n$.

Associated with each compact convex subset $K$ in $\mathbb{R}^n$ is
its \emph{support function} $h_K:S^{n-1}\rightarrow \mathbb{R}$ defined by
\begin{equation}
\label{2.1}
h_K(v)=\max\{x\cdot v: x\in K\}.
\end{equation}
When $K\in \mathcal{K}_o^n$, its \emph{radial function} (with respect to the origin)
$\rho_K:S^{n-1}\rightarrow \mathbb{R}$ is defined by
\begin{equation}
\label{2.2}
\rho_K(u)=\max\{t>0:tu\in K\}.
\end{equation}
If $K\in \mathcal{K}_o^n$, then both $h_K$ and $\rho_K$ are positive.
The volume of $K$ with respect to standard Lebesgue measure is denoted $V(K)$. It is well-known that the volume of $K$ may be computed by integrating
the $n$-th power of the radial function, i.e.,
\begin{equation}
\label{eq volume}
V(K)=\frac{1}{n}\int_{S^{n-1}}\rho_K^n(u)du,
\end{equation}
where $du$ is spherical Lebesgue measure.

We say
that a sequence of convex bodies $K_l$
converges to a compact convex set $K\subset \mathbb{R}^n$ in the \emph{Hausdorff metric} if
\begin{equation*}
\|h_{K_l}-h_K\|\rightarrow 0.
\end{equation*}

For each $h\in C^+(S^{n-1})$, the \emph{Wulff shape} generated by $h$,
denoted $\blb h\brb$, is the convex body defined by
\begin{equation*}
\blb h \brb=\{x\in\mathbb{R}^n: x\cdot v \leq h(v) \text{ for all } v \in S^{n-1}\}.
\end{equation*}
The Wulff shape, also known as the Aleksandrov body, is a key ingredient
in Aleksandrov's volume variational formula, which is essential to the solution
of the classical Minkowski problem. It is easy to see that
\begin{equation}
\label{2.4}
h_{\sblb h \sbrb}\leq h,
\end{equation}
and if $K\in \mathcal{K}_o^n$, then
\begin{equation}
\label{2.5}
\blb h_K \brb =K.
\end{equation}

Given $h_0 \in C^+(S^{n-1}), f \in C(S^{n-1})$, and $\delta>0$, define
$h_t:S^{n-1}\rightarrow (0,\infty)$ for each $t\in (-\delta,\delta)$ by letting
\begin{equation}
\log h_t(v) = \log h_0(v)+tf(v)+o(t,v), v \in S^{n-1},
\end{equation}
where $o(t,\cdot) \in C(S^{n-1})$ satisfies
\[
  \lim_{t\rightarrow 0} \frac{\|o(t,\cdot)\|}{t} = 0.
\]
The family of Wulff shapes generated by $h_t$ is called
a family of logarithmic Wulff shapes generated by $h_0$ and $f$. We sometimes
denote the family $\blb h_t \brb$ by $\blb h_0,f,t\brb$, or simply denote it by $\blb K,f,t\brb$
when $h_0$ is the support function of a convex body $K$.

Assume that $K\in \mathcal{K}_o^n$. The \emph{supporting hyperplane} of
$K$ at $v\in S^{n-1}$ is given by
\begin{equation*}
H_K(v)=\{x\in \mathbb{R}^{n}:x\cdot v = h_K(v)\}.
\end{equation*}
At each boundary point $x\in \partial K$, a vector $v\in S^{n-1}$ is
called an \emph{outer unit normal} of $K$ at $x\in \partial K$ if $x \in H_K(v)$.

Let $\omega\subset S^{n-1}$ be a Borel set. The \emph{radial Gauss image}
$\bm{\alpha}_K(\omega)$, of $K$ at $\omega$, is the set of all outer
unit normals of $K$ at boundary points $\rho_K(u)u$ where $u\in \omega$, i.e.,
\begin{equation}
\bm{\alpha}_K(\omega)=\{v\in S^{n-1}: \text{ there exists }u\in \omega
\text{ such that } \rho_K(u)u\cdot v= h_K(v)\}.
\end{equation}

Let $\eta \subset S^{n-1}$ be a Borel set. The \emph{reverse radial Gauss image}
$\bm{\alpha}_K^*(\eta)$, of $K$ at $\eta$, is the set of all radial directions
 $u\in S^{n-1}$ such that the boundary point $\rho_K(u)u$ has at least one
 element in $\eta$ as its outer unit normal, i.e.,
\begin{equation}
\bm{\alpha}_K^*(\eta)=\{u\in S^{n-1}: \text{there exists } v\in
\eta \text{ such that } \rho_K(u)u\cdot v = h_K(v)\}.
\end{equation}
It was shown in Lemma 2.2.14 of Schneider \cite{schneider2014}
(see also Lemma 2.1 in \cite{HLYZ}) that when $\eta$ is a Borel set,
the set $\bm{\alpha}_K^*(\eta)$ is spherical Lebesgue measurable.

Dual quermassintegrals, which include volume as a special case, are
fundamental geometric invariants in the dual Brunn-Minkowski theory.
For $i=1,\cdots, n$, the $(n-i)$-th \emph{dual quermassintegral}
$\widetilde{W}_{n-i}(K)$ of $K\in \mathcal{K}_o^n$ is proportional
to the mean of $i$-dimensional volumes of the intersections of $K$
and all $i$-dimensional subspaces. That is,
\begin{equation}\label{2.6}
\widetilde{W}_{n-i}(K)=\frac{\omega_n}{\omega_i}\int_{G(n,i)} V_i(K\cap\xi)d\xi,
\end{equation}
where $G(n,i)$ is the Grassmannian manifold of $i$-dimensional linear
subspaces $\xi \subset \mathbb{R}^n$ and the integration is with respect
to the Haar measure on $G(n,i)$. Here $V_i(K\cap \xi)$ denotes the
$i$-th dimensional volume of $K\cap \xi$.
The dual quermassintegrals have the following integral representation (see \cite{L75paci}),
\begin{equation}
\label{2.9}
\widetilde{W}_{n-i}(K)=\frac{1}{n}\int_{S^{n-1}}\rho_K^i(u)du.
\end{equation}
Using this formula, we can define $\widetilde{W}_{n-q}$ for
all $q\in \mathbb{R}$ in the same manner as \eqref{2.9},
\begin{equation}\label{2.7}
\widetilde{W}_{n-q}(K)=\frac{1}{n}\int_{S^{n-1}}\rho_K^q(u)du.
\end{equation}
It is easy to see that the $(n-q)$-th dual quermassintegral is homogeneous
of degree $q$. That is,
\begin{equation*}
\widetilde{W}_{n-q}(cK) = c^q\widetilde{W}_{n-q}(K),
\end{equation*}
for each $K\in\mathcal{K}_o^n$ and $c>0$.

Let $\mu$ be a non-zero finite Borel measure on $S^{n-1}$.
Define the {\em entropy functional} $E_\mu : C^+(\sn) \to \mathbb R$ by
\begin{equation}\label{2.10}
E_\mu(f)=-\frac{1}{|\mu|}\int_{S^{n-1}}\log f(v) \, d\mu(v), \quad f\in C^+(\sn).
\end{equation}
When $f$ is the support function $h_K$ of a convex body $K$, let
\begin{equation}\label{2.11}
E_\mu(K) = E_\mu(h_K).
\end{equation}

\section{The even dual Minkowski problem via maximization}
\label{section maximization problem}

Dual curvature measures in the dual Brunn-Minkowski theory are the counterparts
of curvature measures in the classical Brunn-Minkowski theory. This fundamental
insight was used by Huang-Lutwak-Yang-Zhang \cite{HLYZ} to reformulate the dual
Minkowski problem as the maximization problems described below.

Let $q\in \mathbb{R}$ and $K$ be a convex body in $\mathbb{R}^n$ containing
the origin in its interior. The $q$-th dual curvature measure of $K$, denoted
by $\widetilde{C}_q(K,\cdot)$, can be viewed as a differential of the dual 
quermassintegral $\wt W_{n-q}$ as given by the following variational formula,
\[
\frac{d}{dt} \wt W_{n-q}(\blb K,f,t\brb) \Big|_{t=0} = q \int_{\sn} f(v)\, 
\widetilde{C}_q(K,v),
\]
for $q\neq0$ and $f\in C(\sn)$. There is a similar formula for the case of $q=0$.
The $q$-th dual curvature measure has the following
explicit integral representation,
\begin{equation}
\label{dcm}
\widetilde{C}_q(K,\eta) = \frac{1}{n}\int_{\bm{\alpha}_K^*(\eta)}\rho_K^q(u)du,
\end{equation}
for each Borel set $\eta \subset S^{n-1}$. There is also a Steiner-type formula
associated with dual curvature measures similar to the Steiner formulas for area
and curvature measures, see \cite{HLYZ} for details.

Huang-Lutwak-Yang-Zhang \cite{HLYZ} posed the \emph{dual Minkowski problem},
which asks for necessary and sufficient conditions on a given Borel
measure $\mu$ on $\sn$ so that it is exactly the $q$-th dual curvature
measure of a convex body in $\rn$. Since the unit balls of finite dimensional
Banach spaces are origin-symmetric convex bodies and the dual curvature
measure of an origin-symmetric convex body is even, it is of great
interest to study the following \emph{even dual Minkowski problem}.
\medskip

\noindent
{\bf The Even Dual Minkowski Problem}: {\it
Given an even finite Borel measure $\mu$ on $S^{n-1}$ and $q\in \mathbb{R}$,
 find necessary and sufficient conditions on $\mu$ so that there exists
 an origin-symmetric convex body $K$ in $\rn$ such that
\[
\widetilde{C}_q(K,\cdot) =\mu.
\]
}

When $q=0$, the even dual Minkowski problem is the even Aleksandrov problem,
whose solution was given by Aleksandrov. When $q=n$, the even dual Minkowski problem is the even logarithmic Minkowski problem, whose solution was given by
B\"{o}r\"{o}czky-Lutwak-Yang-Zhang \cite{BLYZ}.

The even dual Minkowski problem when $0<q<n$ was studied in Huang-Lutwak-Yang-Zhang \cite{HLYZ}.
Mass inequalities \eqref{smi1} and \eqref{smi1.1} were shown to be sufficient
for the existence of solutions.
When $0 < q \le 1$, equation \eqref{smi1.1} is both sufficient and necessary
and therefore the even dual Minkowski problem is completely solved.
However, for $1<q<n$, examples
of convex bodies whose dual curvature measures violate \eqref{smi1} were found
in \cite{BH} and \cite{zhao}, showing that \eqref{smi1} is not a necessary condition.
The $q$-th subspace
mass inequality \eqref{smi2} was defined independently in \cite{BH} and \cite{zhao}.
In \cite{BH}, it was shown that, when $1<q<n$, \eqref{smi2} is a necessary condition.
In \cite{zhao}, it was shown that, for $q=2,\cdots, n-1$,
\eqref{smi2} is also a sufficient condition.

It is the aim of this paper to give a complete solution to the even dual Minkowski problem for $1<q<n$. Specifically, we shall prove that,
when $1<q<n$, the $q$-th subspace condition
is both necessary and sufficient for the existence of a solution to the even dual Minkowski problem.

We use the variational method to solve the even dual Minkowski problem.
Here, for completeness, we recall results from \cite{HLYZ}, but
give a slightly different treatment.

The maximization problem whose Euler-Lagrange equation is the equation of the
dual Minkowski problem was formulated in \cite{HLYZ}. To derive the Euler-Lagrange
equation of the maximization problem,
the following variational formula established in \cite{HLYZ} is critical.
If $q\neq 0$, then
\begin{equation}
\label{3.2}
\left.\frac{d}{dt} \widetilde{W}_{n-q}(\blb h_0,f,t \brb)\right|_{t=0}
= q\int_{S^{n-1}}f(v)d\widetilde{C}_q(\blb h_0 \brb, v),
\end{equation}
for each $h_0\in C^+(S^{n-1})$ and $f\in C(S^{n-1})$. Here $\blb h_0,f,t \brb$ is
the logarithmic family of Wulff shapes generated by $h_0$ and $f$,
as defined in Section \ref{section preliminary}. The corresponding
formula when $q=0$ is also given in \cite{HLYZ}.

Let $\mu$ be a non-zero finite even Borel measure on $S^{n-1}$ and $q\neq 0$.
Define the functional $\Phi_\mu : C_e^+(S^{n-1})\rightarrow \mathbb{R}$ to be
\begin{equation}\label{3.3}
\Phi_\mu(f) = E_\mu(f)+\frac{1}{q}\log\widetilde{W}_{n-q}(\blb f\brb),
\quad f\in C_e^+(S^{n-1}).
\end{equation}
\smallskip

\noindent
{\bf Maximization Problem I.}\ For a given non-zero finite even Borel measure $\mu$ on
$\sn$, does there exist an even positive continuous function on $\sn$ that attains the supremum
\begin{equation}
\sup\left\{\Phi_\mu(f):f\in C_e^+(S^{n-1})\right\}?
\end{equation}
\smallskip

Note that the set of support functions of convex bodies in $\mathcal K_e^n$
is a convex sub-cone of  $C_e^+(S^{n-1})$. If the functional $\Phi_\mu$ is restricted
to this sub-cone and the support function of a convex body is identified with the convex
body, the functional $\Phi_\mu$ can be treated as a functional on $\mathcal K_e^n$,
$\Phi_\mu : \mathcal K_e^n \rightarrow \mathbb{R}$, given by
\begin{equation}\label{3.4}
\Phi_\mu(K) = E_\mu(K)+\frac{1}{q}\log\widetilde{W}_{n-q}(K),
\quad K \in \mathcal K_e^n.
\end{equation}
In particular,
\begin{equation}\label{3.5}
\Phi_\mu(K) = \Phi_\mu(h_K).
\end{equation}

This leads to the following variational problem.
\smallskip

\noindent
{\bf Maximization Problem II.}\ For a given non-zero finite even Borel measure $\mu$ on
$\sn$, does there exist a convex body in $\mathcal K_e^n$ that attains the supremum,
\begin{equation}
\sup\left\{\Phi_\mu(K) : K \in \mathcal K_e^n  \right\}?
\end{equation}
\smallskip

The following lemma shows that, if we identify a convex body $K$ with its support function $h_K$, then a solution to
Maximization Problem II is  a solution to
Maximization Problem I.

\begin{lemma}\label{3.12}
Let $\mu$ be a non-zero even finite Borel measure
on $S^{n-1}$ and $q$ a non-zero real number. If there exists $K_0\in\mathcal{K}_e^n$
such that
\begin{equation}\label{3.6}
\Phi_\mu(K_0)=\sup\left\{\Phi_\mu(K) : K\in \mathcal{K}_e^n\right\},
\end{equation}
then
\begin{equation}\label{3.7}
 \Phi_\mu(h_{K_0})=\sup\left\{\Phi_\mu(f):f\in C_e^+(S^{n-1})\right\}.
\end{equation}
\end{lemma}

\begin{proof}
Let $f\in C_e^+(S^{n-1})$. Note that $\blb f \brb\in \mathcal{K}_e^n$.
By \eqref{3.5} and \eqref{3.6},
\begin{equation*}
	\Phi_\mu(h_{K_0})=\Phi_\mu(K_0)\geq \Phi_\mu(\blb f \brb).
\end{equation*}
By \eqref{2.4} and \eqref{2.10},
\begin{equation}\label{3.8}
E_\mu(\blb f \brb)\geq E_\mu(f),
\end{equation}
for each $f\in C_e^+(S^{n-1})$.
By \eqref{3.8},
\begin{equation*}
	\begin{aligned}
\Phi_\mu(\blb f \brb)&=E_\mu(\blb f \brb)+\frac{1}{q}\log \widetilde{W}_{n-q}(\blb f \brb)\\
		&\geq E_\mu(f)+\frac{1}{q}\log \widetilde{W}_{n-q}(\blb f\brb)\\
		&=\Phi_\mu(f).
	\end{aligned}
\end{equation*}
Hence, $\Phi_\mu(h_{K_0})\geq \Phi_\mu(f)$ for all $f\in C_e^+(S^{n-1})$, proving the lemma.
\end{proof}

The next lemma shows that a solution to Maximization Problem I is a solution
to the even dual Minkowski problem.

\begin{lemma}\label{3.9}
Let $\mu$ be a non-zero even finite Borel measure
on $S^{n-1}$ and $q\neq 0$. If there exists $K_0\in\mathcal{K}_e^n$ such that
\begin{equation*}
	\Phi_\mu(h_{K_0})=\sup\left\{\Phi_\mu(f):f\in C_e^+(S^{n-1})\right\},
\end{equation*}
then there exists $c>0$ such that
\begin{equation*}
	\mu=\widetilde{C}_q(cK_0,\cdot).
\end{equation*}
\end{lemma}
\begin{proof}
Since the $(n-q)$-th dual quermassintegral is homogeneous of degree $q$,
 we may choose $c>0$ so that
\begin{equation}\label{3.10}
\widetilde{W}_{n-q}(cK_0)=c^q\widetilde{W}_{n-q}(K_0)=|\mu|.
\end{equation}
Since $\Phi_\mu$ is homogeneous of degree $0$,
\begin{equation}\label{3.11}
\Phi_\mu(h_{cK_0})=\Phi_\mu(h_{K_0})
=\sup\left\{\Phi_\mu(f):f\in C_e^+(S^{n-1})\right\}.
\end{equation}
For each $g\in C_e(S^{n-1})$, define $h_t\in C_e^+(S^{n-1})$ by
\begin{equation*}
	\log h_t(v)=\log h_{cK_0}(v)+tg(v), \forall v\in S^{n-1}.
\end{equation*}
By \eqref{3.11},
\begin{equation*}
	\Phi_\mu(h_0)=\Phi_\mu(h_{cK_0})\leq \Phi_\mu(h_t).
\end{equation*}
Hence, by the definition of $\Phi_\mu$ and $E_\mu$, \eqref{3.2}, and \eqref{3.10},
\begin{equation*}
\begin{aligned}
0 &= \left.\frac{d}{dt}\Phi_\mu(h_t)\right|_{t=0}\\
  &=\left.\frac{d}{dt}\left(E_\mu(h_t)+\frac{1}{q}\log
   \widetilde{W}_{n-q}([cK_0,g,t])\right)\right|_{t=0}\\
  &=-\frac{1}{|\mu|}\int_{S^{n-1}}g(v)\, d\mu(v)+\frac{1}{\widetilde{W}_{n-q}(cK_0)}
   \int_{S^{n-1}}g(v) \, d\widetilde{C}_q(cK_0,v)\\
  &=\frac{1}{|\mu|}\left(-\int_{S^{n-1}}g(v) \, d\mu(v)+\int_{S^{n-1}}g(v) \,
   d\widetilde{C}_q(cK_0,v)\right).
	\end{aligned}
\end{equation*}
Since this holds for any $g\in C_e(S^{n-1})$, it follows that
\begin{equation*}
	\mu = \widetilde{C}_q(cK_0,\cdot).
\end{equation*}
\end{proof}

By Lemmas \ref{3.9} and \ref{3.12}, a solution to Maximization Problem II
is a solution to the even dual Minkowski problem. This is stated formally in the
following lemma.

\begin{lemma}\label{3.13}
Let $\mu$ be a non-zero even finite Borel measure on $S^{n-1}$ and $q\neq 0$.
If there exists $K_0\in\mathcal{K}_e^n$ such that
\begin{equation*}
	\Phi_\mu(K_0)=\sup\left\{\Phi_\mu(K) : K\in \mathcal{K}_e^n\right\},
\end{equation*}
then there exists $c>0$ such that
\begin{equation*}
	\mu=\widetilde{C}_q(cK_0,\cdot).
\end{equation*}
\end{lemma}

Therefore, to solve the even dual Minkowski problem,
it suffices to solve Maximization Problem II.
Solving Maximization Problem II requires delicate estimates for the functional $E_\mu$ and
the quermassintegral $\wt W_{n-q}$, which will be dealt with in the next two sections.

\section{Estimates for the entropy functional $E_\mu$}
\label{est1}

In this section, we will estimate the functional $E_\mu$
under the assumption that $\mu$ satisfies the subspace mass inequality \eqref{smi2}.

Let $q>0$ be a real number. Recall that a non-zero finite even
Borel measure $\mu$ on $S^{n-1}$ satisfies the
\emph{$q$-th subspace mass inequality} if
\begin{equation}
\label{4.1}
\frac{\mu(\xi\cap S^{n-1})}{|\mu|}<
\begin{cases}
\frac{i}{q}, & i<q,\\
1, & i\geq q,
\end{cases}
\end{equation}
for each $i$-dimensional subspace $\xi\subset \mathbb{R}^n$.
We assume,
for the rest of this section, that $1<q<n$ and $\mu$ is a non-zero finite even
Borel measure on $S^{n-1}$ satisfying the subspace mass inequality \eqref{4.1}.

The key technique for estimating $E_\mu$ is to use an appropriate
spherical partition. This approach was first introduced in \cite{BLYZ}.
Let $e_1,\cdots, e_n$ be an orthonormal basis in $\mathbb{R}^n$.
For each $\delta\in (0,\frac{1}{\sqrt{n}})$, define
a partition $\{\Omega_{i,\delta}\}_{i=1}^n$ of $\sn$, where
\begin{equation}\label{4.2}
\Omega_{i,\delta}=\{v\in S^{n-1}: |v\cdot e_i|\geq \delta \text{ and }|v\cdot e_j|<\delta
\text{ for all } j>i\},
\end{equation}
for $i=1,\cdots, n$ and $\delta>0$.

For notational simplicity, let
\begin{equation*}
\xi_i = \Span\{e_1,\cdots, e_i\}, \quad i=1,\cdots, n,
\end{equation*}
and
$\xi_0 = \{0\}$.

It was shown in \cite{BLYZ} that
for any non-zero finite Borel measure $\mu$ on $S^{n-1}$,
\begin{equation}\label{4.3}
\lim_{\delta\rightarrow 0^+}\mu(\Omega_{i,\delta})
=\mu((\xi_i\setminus \xi_{i-1})\cap S^{n-1}),
\end{equation}
and, therefore,
\begin{equation}\label{4.4}
\lim_{\delta\rightarrow 0^+}\big(\mu(\Omega_{1,\delta}) + \cdots +
\mu(\Omega_{i,\delta})\big)=\mu(\xi_i\cap S^{n-1}).
\end{equation}
We also will need the following elementary lemma.

\begin{lemma}\label{4.5}
Let $\lambda_1,\cdots, \lambda_n, x_1,\cdots, x_n\in \mathbb{R}$.
Assume that $x_1\leq x_2\leq \cdots\leq x_n$, $\lambda_i\geq 0$,
and $\lambda_1+\cdots+\lambda_n=1$.
Suppose there exists $\sigma_1,\cdots, \sigma_n\in \mathbb{R}$
with $\sigma_n=1$ such that
\begin{equation}\label{4.8}
\lambda_1+\cdots+\lambda_i\leq \sigma_i, \quad i=1,\cdots, n.
\end{equation}
Then
\begin{equation*}
\sum_{i=1}^n\lambda_i x_i\geq \sum_{i=1}^{n}(\sigma_i-\sigma_{i-1})x_i,
\end{equation*}
where $\sigma_0=0$.
\end{lemma}

\begin{proof}
Let
\begin{equation*}
s_i=\lambda_1+\cdots+\lambda_i
\end{equation*}
for $i=1,\cdots, n$ and $s_0=0$. Note that $s_n=1$ and $s_i\leq \sigma_i$. Observe that
\begin{equation}\label{4.6}
\lambda_i=s_i-s_{i-1}, i=1,\cdots, n.
\end{equation}

By \eqref{4.6},
\begin{equation}\label{4.7}
\begin{aligned}
\sum_{i=1}^n\lambda_ix_i & = \sum_{i=1}^n(s_i-s_{i-1})x_i\\
	&= \sum_{i=1}^n s_ix_i-\sum_{i=1}^{n-1}s_i x_{i+1}\\
	&= \sum_{i=1}^{n-1}s_i(x_i-x_{i+1})+x_n.
\end{aligned}
\end{equation}
Since $x_1\leq x_2\leq \cdots \leq x_n$, it follows
by equations \eqref{4.8}
and \eqref{4.7}, that
\begin{equation*}
\begin{aligned}
\sum_{i=1}^n \lambda_i x_i &\geq \sum_{i=1}^{n-1}\sigma_i (x_i-x_{i+1})+x_n\\
	&=\sum_{i=1}^{n-1}\sigma_ix_i-\sum_{i=2}^{n}\sigma_{i-1}x_i+x_n\\
	&=\sum_{i=1}^n\sigma_ix_i-\sum_{i=1}^{n}\sigma_{i-1}x_i\\
	&=\sum_{i=1}^n(\sigma_i-\sigma_{i-1})x_i.
\end{aligned}
\end{equation*}
\end{proof}

The following lemma provides the key estimate for $E_\mu$.

\begin{lemma}\label{4.9}
  Let $1<q<n$ be a real number and  $Q_l$, $l = 1, 2, \ldots$, be a sequence of ellipsoids given by
\begin{equation}
Q_l=\Big\{x\in\mathbb{R}^n: \frac{|x\cdot e_{1l}|^2}{a_{1l}^2}+\cdots
 +\frac{|x\cdot e_{nl}|^2}{a_{nl}^2}\leq 1\Big\},
\end{equation}
  where $\{e_{1l},\cdots, e_{nl}\}$ is a sequence of orthonormal bases of $\R^n$ converging to an orthonormal basis $\{e_1,\dots,e_n\}$ and
$(a_{1l},\cdots, a_{nl})$ a sequence of $n$-tuples satisfying
$a_{1l}\leq a_{2l}\leq \cdots \leq a_{nl}$ such that
$a_{nl}>\varepsilon_0$, for all $l$, for some $\epsilon_0 > 0$.

If $\mu$ is a non-zero even finite Borel measure on $S^{n-1}$
that satisfies the $q$-th subspace mass inequality \eqref{4.1},
then there exists $t_0, \delta_0, l_0>0$ such that for each $l>l_0$ we have
\begin{enumerate}
	\item if $1<q<n-1$,
	\begin{equation}\label{4.10}
		E_\mu(Q_l)\leq -\frac{1}{q}\log(a_{1l}\cdots a_{\lfloor q\rfloor l})
 -\frac{q-\lfloor q\rfloor}{q}\log a_{\lfloor q\rfloor+1,l}+t_0\log a_{1l}
  +c(\varepsilon_0,t_0,\delta_0);
	\end{equation}
	\item if $n-1\leq q<n$,
	\begin{equation}\label{4.11}
		E_\mu(Q_l)\leq -\frac{1}{q}\log(a_{1l}\cdots a_{n-1,l})+t_0\log a_{1l}
+c(q,\varepsilon_0,t_0,\delta_0).
	\end{equation}
\end{enumerate}
\end{lemma}

\begin{proof}
Define $\Omega_{i,\delta}$ as in \eqref{4.2}
with respect to $e_1,\cdots, e_n$. For $\delta<\frac{1}{\sqrt{n}}$, let
\begin{equation*}
\lambda_{i,\delta}=\frac{\mu(\Omega_{i,\delta})}{|\mu|}.
\end{equation*}
Note that $\lambda_{1,\delta}+\cdots+\lambda_{n,\delta}=1$.
By \eqref{4.4} and \eqref{4.1}, we have
\begin{equation}
	\lim_{\delta\rightarrow 0^+}(\lambda_{1,\delta}+\cdots+\lambda_{i,\delta})
=\frac{\mu(\xi_i\cap S^{n-1})}{|\mu|}<\min\left\{\frac{i}{q},1\right\}, \qquad i=1,2,\dots, n-1.
\end{equation}
Since the inequality is strict, we may choose $t_0,\delta_0>0$ such that
\begin{equation}	\label{4.12}
\lambda_{1,\delta_0}+\cdots+\lambda_{i,\delta_0}<\min\left\{\frac{i}{q},1\right\}-t_0, \qquad i=1,2,\dots,n-1.
\end{equation}
Since $\lim_{l\rightarrow \infty}e_{il}=e_{i}$ for every $i =1,\cdots, n$,
we may choose $l_0>0$ such that
\begin{equation}\label{4.13}
|e_{i}-e_{il}|<\frac{\delta_0}{2},
\text{ for every } i=1, \cdots, n \text{ and } l>l_0.
\end{equation}

We assume for the rest of the proof that $l>l_0$ so that \eqref{4.13} is always satisfied.
For each $v\in \Omega_{i,\delta_0}$, by the definition of $Q_l$
and $\Omega_{i,\delta_0}$, and \eqref{4.13},
\begin{equation*}
 h_{Q_l}(v)\geq a_{il}|e_{il}\cdot v|\geq a_{il}\left(|e_i\cdot v|-|e_i-e_{il}|\right)
 \geq a_{il}\frac{\delta_0}{2}.
\end{equation*}
This and the partition \eqref{4.2} imply
\begin{equation}\label{4.14}
E_\mu(Q_l)\leq -\sum_{i=1}^n\frac{1}{|\mu|}\int_{\Omega_{i,\delta_0}}
\Big(\log a_{il}+\log\frac{\delta_0}{2}\Big) \, d\mu(v) =-\log\frac{\delta_0}{2}-
\sum_{i=1}^n\lambda_{i,\delta_0}\log a_{il}.
\end{equation}
Let $\lambda_i =\lambda_{i,\delta_0}$, $x_i=\log a_{il}$. Define $\sigma_0 = 0$, $\sigma_n=1$, and
\begin{equation*}
\sigma_i = \min\left\{\frac{i}{q}, 1\right\} - t_0,
\end{equation*}
for $i = 1, \dots, n-1$. It follows that when $1<q<n-1$,
\[
\sigma_i - \sigma_{i-1} = \begin{cases}
\frac1q - t_0 & i =1\\
\frac1q &1<i\le \lfloor q \rfloor \\
1-\frac{\lfloor q \rfloor}q &i=\lfloor q \rfloor + 1 \\
0 &\lfloor q \rfloor + 1 <i<n \\
t_0 &i=n,
\end{cases}
\]
and when $n-1 \leq q<n$,
\[
\sigma_i - \sigma_{i-1} = \begin{cases}
\frac1q - t_0 & i =1\\
\frac1q &1<i\le n-1 \\
1-\frac{n-1}q+t_0 &i=n.
\end{cases}
\]
By Lemma \ref{4.5}, whose assumptions are implied by \eqref{4.12} and
the fact that $a_{1l}\leq \cdots\leq a_{nl}$, we have when $1<q<n-1$
\begin{align*}
\sum_{i=1}^n &\lambda_{i,\delta_0} \log a_{il}
 \ge \sum_{i=1}^n (\sigma_i - \sigma_{i-1}) \log a_{il} \\
&=\Big(\frac1q-t_0\Big) \log a_{1l} + \sum_{i=2}^{\lfloor q \rfloor} \frac1q \log a_{il}
  +\Big(1-\frac{\lfloor q \rfloor}q\Big) \log a_{\lfloor q \rfloor + 1,l} +t_0 \log a_{nl}.
\end{align*}
Using the same argument, it can be seen that the above equation also works when $n-1\le q<n$.
By this and \eqref{4.14}, we obtain
\begin{equation}\label{4.15}
	\begin{aligned}
		E_\mu(Q_l)&\leq -\log\frac{\delta_0}{2}-\Big(\frac{1}{q}-t_0\Big)\log a_{1l}
-\sum_{i=2}^{\lfloor q \rfloor}\frac{1}{q}\log a_{il}
-\Big(1-\frac{\lfloor q \rfloor}{q}\Big)\log a_{\lfloor q \rfloor+1,l}-t_0\log a_{nl}\\
		&=-\log\frac{\delta_0}{2}+t_0\log a_{1l}
-\frac{1}{q}\log (a_{1l}\cdots a_{\lfloor q \rfloor l})
-\frac{q-\lfloor q \rfloor}{q}\log a_{\lfloor q \rfloor+1,l}-t_0\log a_{nl}.\\
	\end{aligned}
\end{equation}
When $1<q<n-1$, equation \eqref{4.15} and the fact that $a_{nl}>\varepsilon_0$
give \eqref{4.10}. When $n-1 \leq q <n$, we have $\lfloor q\rfloor=n-1$. Again,
the fact that $a_{nl}>\varepsilon_0$ and \eqref{4.15} give \eqref{4.11}.
\end{proof}

\section{Estimates for dual quermassintegrals}
\label{5.1}

Solving the even dual Minkowski problem when $1<q<n$ requires
estimates for dual quermassintegrals, which are in general difficult to establish.
One indication of this is that, when $q$ is an integer,
the dual quermassintegrals involve lower dimensional cross sections of a convex body and are defined using integration over Grassmannians, as shown by \eqref{2.6}.
This is a new obstacle that is not present in the logarithmic Minkowski problem.
To overcome this, we employ two techniques introduced in \cite{HLYZ}. One is to use
 general spherical coordinates to decompose the dual quermassintegral into a sum of integrals and estimating each integral separately.
The other is to choose the right barrier
convex body that will yield optimal estimates.

When $n-1 \leq q < n$, we use a Cartesian product of an ellipsoid and a ball as the barrier.
The following lemma was proved in \cite{zhao}.

\begin{lemma}
\label{5.2}
Suppose $1\leq k \leq n-1$ is an integer and $k<q\leq n$.
Let $e_1, \cdots, e_n$ be an orthonormal basis in $\mathbb{R}^n$
and $a_1,\cdots, a_k>0$. Define
\begin{equation*}
T=\left\{x\in \mathbb{R}^n:\frac{|x\cdot e_1|^2}{a_1^2}+\cdots
+\frac{|x\cdot e_k|^2}{a_k^2}\leq 1, \ |x\cdot e_{k+1}|^2+\cdots
+|x\cdot e_{n}|^2\leq 1\right\}.
\end{equation*}
Then
\begin{equation*}
\widetilde{W}_{n-q}(T)\leq c(n,k,q)\, a_1\cdots a_k.
\end{equation*}
\end{lemma}

Although Lemma \ref{5.2} is enough for solving the even dual Minkowski problem when
$q\in \{1, 2,\ldots, n-1\}$
 (see \cite{zhao}) or when $n-1\leq q\leq n$
(see Lemma \ref{5.3} in Section \ref{6.0}),
stronger estimates are needed for non-integer $q\in (1,n-1)$.
This requires a more careful choice of the barrier convex body and more involved calculations to obtain sufficiently sharp estimates for the dual quermassintegrals of this body.
The rest of this section will focus on deriving these estimates.

For the rest of this section
we always assume that the dimension $n$ is at least $3$.

We recall the definition of general spherical coordinates. Given $1 \le l \le n-1$, decompose $\rn = \mathbb R^l \times \mathbb R^{n-l}$.
The general spherical coordinates are given by
\begin{equation}\label{5.1.1}
u=(w\cos\phi, v\sin\phi) \in \sn, \ w\in S^{l-1},\ v\in S^{n-l-1},\ \phi\in [0,\pi/2].
\end{equation}
Denote by $du, dw, dv$ the spherical Lebesgue measures on $\sn, S^{l-1}, S^{n-l-1}$, respectively.
These satisfy (see, for example, \cite{GZ})
\begin{equation}\label{5.1.2}
du = \cos^{l-1}\phi \sin^{n-l-1}\phi \, dw dv d\phi.
\end{equation}

We need spherical coordinates system more general than \eqref{5.1.1}.
Let $e_1, \cdots, e_n$ be an orthonormal basis in $\mathbb{R}^n$. Suppose $k$ and $j$
are two positive integers such that $k+j<n$. Write
$\mathbb{R}^n = \mathbb{R}^k\times \mathbb{R}^j \times \mathbb{R}^{n-k-j}$.
For each $u\in S^{n-1}$, we consider the following general spherical coordinates,
\begin{equation}\label{5.4}
u=(u_1\cos \phi \cos \theta, u_2\cos \phi\sin \theta, u_3\sin\phi),
\end{equation}
where $u_1\in S^{k-1}\subset \mathbb{R}^k, u_2\in S^{j-1}\subset\mathbb{R}^j,
u_3\in S^{n-k-j-1}\subset \mathbb{R}^{n-k-j}$, and $\theta,\phi \in [0,\pi/2]$.

The following lemma expresses the spherical Lebesgue measure on $S^{n-1}$
in terms of spherical Lebesgue measures on lower dimensional spheres.

\begin{lemma}
For each $u\in S^{n-1}$, if we write $u$ as in \eqref{5.4}, then
\begin{equation}\label{5.5}
du = \cos^{k+j-1}\phi \sin^{n-k-j-1}\phi \cos^{k-1}\theta\sin^{j-1}\theta \,
du_1\,du_2\,du_3\,d\phi\, d\theta.
\end{equation}
\end{lemma}

\begin{proof}
Let $w=(u_1\cos \theta, u_2\sin\theta)$. Then $w\in S^{k+j-1}$. Since now,
\begin{equation*}
	u=(w\cos\phi, u_3\sin \phi),
\end{equation*}
by \eqref{5.1.2},
\begin{equation}\label{5.6}
du = \cos^{k+j-1}\phi \sin^{n-k-j-1}\phi \, dw\,du_3\,d\phi.
\end{equation}
By \eqref{5.1.2} again and $w=(u_1\cos \theta, u_2\sin \theta)$,
\begin{equation}\label{5.7}
dw= \cos^{k-1}\theta\sin^{j-1}\theta \, du_1\,du_2\,d \theta.
\end{equation}
Substituting \eqref{5.7} into \eqref{5.6} gives us \eqref{5.5}.
\end{proof}

For the purpose of estimating dual quermassintegrals in the even dual Minkowski problem,
we focus on the special case where $j=1$ and $k\in \{1,\cdots, n-2\}$.
In this case, equation \eqref{5.5} becomes
\begin{equation}
\label{eq local 18} du = \cos^{k}\phi \sin^{n-k-2}\phi \cos^{k-1}\theta \,
du_1\,du_2\,du_3\,d\phi\, d\theta,
\end{equation}
where
\begin{equation}
\label{5.12} u = (u_1\cos\phi \cos \theta,u_2\cos\phi \sin\theta, u_3\sin \phi),
\end{equation}
with $u_1\in S^{k-1}$, $u_2\in S^{0}$, $u_3\in S^{n-k-2}$, and $\phi,\theta\in [0,\pi/2]$.

Let $a_1, \cdots, a_{k+1}$ be $k+1$ real numbers such that
$0<a_1\leq a_2\leq \cdots \leq a_{k+1}<1$. Let $G$ be the Cartesian product of
an ellipsoid, a line segment, and a ball in lower dimensional subspaces, i.e.,
\begin{equation}
\label{5.8}
G=\Big\{x\in \mathbb{R}^n:\frac{|x\cdot e_1|^2}{a_1^2}+\cdots+\frac{|x\cdot e_k|^2}{a_k^2}\leq 1,\
 |x\cdot e_{k+1}|\leq a_{k+1},\ |x\cdot e_{k+2}|^2+\cdots +|x\cdot e_n|^2\leq 1\Big\}.
\end{equation}
It turns out that $G$ is exactly the barrier convex body that will
provide the estimate needed to solve the even dual Minkowski
problem when $q\in (1,n-1)$ is not necessarily an integer.

Let $\bar{G}\subset \mathbb{R}^k$ be the ellipsoid in the Cartesian product body $G$,
i.e.,
\begin{equation}
\label{5.9}
\bar{G}=\Big\{x\in \mathbb{R}^k: \frac{|x\cdot e_1|^2}{a_1^2}+\cdots+
\frac{|x\cdot e_k|^2}{a_k^2}\leq 1 \Big\}.
\end{equation}
The radial function of $\bar G$ is given by
\begin{equation}\label{5.10}
\rho_{\bar G}(u_1) = \Big(\frac{|u_1\cdot e_1|^2}{a_1^2}+
\cdots+\frac{|u_1\cdot e_k|^2}{a_k^2}\Big)^{-\frac12}.
\end{equation}

\begin{lemma}
\label{5.11}
Let $n\geq 3$ and $1\leq k\leq n-2$ be integers. Let $0<a_1\leq a_2\leq \cdots \leq a_{k+1}<1$
be $k+1$ real numbers. Define $G$ and $\bar{G}$ as in \eqref{5.8} and \eqref{5.9} respectively.
If unit vectors $u\in S^{n-1}$ are parameterized as in \eqref{5.12}, i.e.,
\begin{equation}
\label{5.13}
u = (u_1\cos\phi \cos \theta,u_2\cos\phi \sin\theta, u_3\sin \phi),
\end{equation}
where $u_1\in S^{k-1}$, $u_2\in S^{0}$, $u_3\in S^{n-k-2}$,
and $\phi, \theta\in [0,\pi/2]$, then the radial function of $G$ is given by
\begin{equation}\label{5.14}
\rho_G(u)=
\begin{cases}
	\dfrac{\rho_{\bar{G}}(u_1)}{\cos \phi \cos \theta},
& \text{if } \theta\in \left[0,\arctan \frac{a_{k+1}}{\rho_{\bar{G}}(u_1)}\right]
\text{ and }\phi\in \left[0,\arctan \frac{\cos\theta}{\rho_{\bar{G}}(u_1)}\right];\\
	\dfrac{1}{\sin \phi},
&\text{if }\theta \in \left[0,\arctan \frac{a_{k+1}}{\rho_{\bar{G}}(u_1)}\right]
\text{ and }\phi\in \left(\arctan \frac{\cos\theta}{\rho_{\bar{G}}(u_1)},\pi/2\right];\\
	\dfrac{a_{k+1}}{\sin\theta\cos\phi},
&\text{if } \theta \in \left(\arctan\frac{a_{k+1}}{\rho_{\bar{G}}(u_1)},\pi/2\right]
\text{ and }\phi\in \left[0,\arctan\frac{\sin\theta}{a_{k+1}}\right];\\
	\dfrac{1}{\sin \phi},
&\text{if }\theta \in \left(\arctan\frac{a_{k+1}}{\rho_{\bar{G}}(u_1)},\pi/2\right]
\text{ and }\phi\in \left(\arctan\frac{\sin\theta}{a_{k+1}},\pi/2\right].
\end{cases}
\end{equation}
\end{lemma}

\begin{proof}
By the definition of radial function, \eqref{5.8}, and \eqref{5.10}, we have
\begin{equation}\label{5.15}
\begin{aligned}
	\rho_G(u)&= \max\{t>0: tu\in G\}\\
	& = \max\{t>0:(tu_1\cos\phi\cos\theta, tu_2\cos \phi \sin \theta, tu_3 \sin\phi)\in G\}\\
	&=\max\left\{t>0:t\leq \frac{\rho_{\bar{G}}(u_1)}{\cos\phi\cos\theta},
t\leq \frac{a_{k+1}}{\cos \phi \sin\theta}, t\leq \frac{1}{\sin\phi}\right\}\\
	&= \min\left\{\frac{\rho_{\bar{G}}(u_1)}{\cos\phi\cos\theta},
\frac{a_{k+1}}{\cos\phi\sin\theta},\frac{1}{\sin\phi}\right\}.
\end{aligned}
\end{equation}

We first consider the case when
$\theta\in \left[0,\arctan \frac{a_{k+1}}{\rho_{\bar{G}}(u_1)}\right]$.
Note that in this case,
\begin{equation*}
\frac{\rho_{\bar{G}}(u_1)}{\cos\phi\cos\theta}\leq \frac{a_{k+1}}{\cos\phi\sin\theta}.
\end{equation*}
Hence, by \eqref{5.15},
\begin{equation*}
\rho_G(u)=\min\left\{\frac{\rho_{\bar{G}}(u_1)}{\cos\phi\cos\theta},\frac{1}{\sin\phi}\right\}.
\end{equation*}
When $\tan\phi\leq \frac{\cos\theta}{\rho_{\bar{G}}(u_1)}$
(or $\phi\in \left[0,\arctan\frac{\cos\theta}{\rho_{\bar{G}}(u_1)}\right]$), we have
\begin{equation*}
\rho_G(u)=\frac{\rho_{\bar{G}}(u_1)}{\cos \phi\cos\theta}.
\end{equation*}
When $\tan\phi>\frac{\cos\theta}{\rho_{\bar{G}}(u_1)}$
(or $\phi\in \left(\arctan\frac{\cos\theta}{\rho_{\bar{G}}(u_1)},\pi/2\right]$),
we have
\begin{equation*}
\rho_G(u)=\frac{1}{\sin\phi}.
\end{equation*}

We now consider the case when
$\theta \in \left(\arctan \frac{a_{k+1}}{\rho_{\bar{G}}(u_1)}, \pi/2\right]$.
Note that in this case,
\begin{equation*}
\frac{\rho_{\bar{G}}(u_1)}{\cos\phi\cos\theta}\geq \frac{a_{k+1}}{\cos\phi\sin\theta}.
\end{equation*}
Hence, by \eqref{5.15},
\begin{equation*}
\rho_G(u)=\min\left\{\frac{a_{k+1}}{\cos \phi\sin\theta},\frac{1}{\sin\phi}\right\}.
\end{equation*}
When $\tan\phi\leq \frac{\sin\theta}{a_{k+1}}$
(or $\phi\in \left[0,\arctan\frac{\sin\theta}{a_{k+1}}\right]$), we have
\begin{equation*}
\rho_G(u)=\frac{a_{k+1}}{\sin\theta\cos\phi}.
\end{equation*}
When $\tan\phi>\frac{\sin\theta}{a_{k+1}}$
(or $\phi\in \left(\arctan\frac{\sin\theta}{a_{k+1}},\pi/2\right]$), we have
\begin{equation*}
\rho_G(u)=\frac{1}{\sin\phi}.
\end{equation*}
\end{proof}

The next lemma gives an explicit formula for the $(n-q)$-th
dual quermassintegral of the convex body $G$.
\begin{lemma}
\label{5.16}
Let $n\geq 3$ and $1\leq k\leq n-2$ be integers.
Let $0<a_1\leq a_2\leq \cdots \leq a_{k+1}<1$ be $k+1$ real numbers.
Define $G$ and $\bar{G}$ as in \eqref{5.8} and \eqref{5.9} respectively. Then
\begin{equation*}
\wt W_{n-q}(G)=\frac 2n (n-k-1)\omega_{n-k-1}(I_1+I_2+I_3+I_4),
\end{equation*}
where
\begin{equation}
\label{5.17}
\begin{aligned}
I_1 &= \int_{S^{k-1}}du_1\int_{0}^{\arctan \frac{a_{k+1}}{\rho_{\bar{G}}(u_1)}}d\theta
\int_{0}^{\arctan \frac{\cos\theta}{\rho_{\bar{G}}(u_1)}}\rho_{\bar{G}}^q(u_1)
\cos^{k-1-q}\theta\cos^{k-q}\phi \sin^{n-k-2}\phi\, d\phi,\\
I_2 &= \int_{S^{k-1}}du_1\int_{0}^{\arctan \frac{a_{k+1}}{\rho_{\bar{G}}(u_1)}}d\theta
\int_{\arctan\frac{\cos \theta}{\rho_{\bar{G}}(u_1)}}^{\pi/2}\cos^{k-1}\theta
\cos^k\phi\sin^{n-k-q-2}\phi \, d\phi,\\
I_3 &= \int_{S^{k-1}}du_1\int_{\arctan\frac{a_{k+1}}{\rho_{\bar{G}}(u_1)}}^{\pi/2}d\theta
\int_{0}^{\arctan \frac{\sin\theta}{a_{k+1}}}a_{k+1}^q\cos^{k-1}\theta\sin^{-q}\theta
\cos^{k-q}\phi\sin^{n-k-2}\phi\, d\phi,\\
I_4 &= \int_{S^{k-1}}du_1\int_{\arctan\frac{a_{k+1}}{\rho_{\bar{G}}(u_1)}}^{\pi/2}d\theta
\int_{\arctan \frac{\sin\theta}{a_{k+1}}}^{\pi/2}\cos^{k-1}\theta
\cos^{k}\phi\sin^{n-k-q-2}\phi\, d\phi.
\end{aligned}
\end{equation}
\end{lemma}

\begin{proof}
By Lemma \ref{5.11}, we have
\begin{align*}
&\hskip -12pt \int_{S^{n-1}}\rho_G^q(u)du \\
= &\int_{S^{k-1}}du_1\int_{S^0}du_2\int_{S^{n-k-2}}du_3
\int_{0}^{\arctan \frac{a_{k+1}}{\rho_{\bar{G}}(u_1)}}d\theta
\int_{0}^{\arctan \frac{\cos\theta}{\rho_{\bar{G}}(u_1)}} \\
& \hskip 255pt \rho_{\bar{G}}^q(u_1)\cos^{k-1-q}\theta\cos^{k-q}\phi \sin^{n-k-2}\phi\, d\phi\\
&+ \int_{S^{k-1}}du_1\int_{S^0}du_2\int_{S^{n-k-2}}du_3
\int_{0}^{\arctan \frac{a_{k+1}}{\rho_{\bar{G}}(u_1)}}d\theta
\int_{\arctan\frac{\cos \theta}{\rho_{\bar{G}}(u_1)}}^{\pi/2}\cos^{k-1}\theta
\cos^k\phi\sin^{n-k-q-2}\phi\, d\phi\\
&+ \int_{S^{k-1}}du_1\int_{S^0}du_2\int_{S^{n-k-2}}du_3
\int_{\arctan\frac{a_{k+1}}{\rho_{\bar{G}}(u_1)}}^{\pi/2}d\theta
\int_{0}^{\arctan \frac{\sin\theta}{a_{k+1}}}\\
& \hskip 240pt a_{k+1}^q\cos^{k-1}\theta\sin^{-q}\theta\cos^{k-q}\phi\sin^{n-k-2}\phi\, d\phi\\
&+ \int_{S^{k-1}}du_1\int_{S^0}du_2\int_{S^{n-k-2}}du_3
\int_{\arctan\frac{a_{k+1}}{\rho_{\bar{G}}(u_1)}}^{\pi/2}d\theta
\int_{\arctan \frac{\sin\theta}{a_{k+1}}}^{\pi/2}\cos^{k-1}\theta
\cos^{k}\phi\sin^{n-k-q-2}\phi\, d\phi.
\end{align*}
The desired result follows immediately from \eqref{2.7}, integrating with respect to
$u_2$ and $u_3$, and the fact that the surface areas of
$S^0$ and $S^{n-k-2}$ are 2 and $(n-k-1)\omega_{n-k-1}$, respectively.
\end{proof}

The next two lemmas simplify the integrals in \eqref{5.17}.

\begin{lemma}\label{5.18}
Let $n\geq 3$ and $1\leq k\leq n-2$ be integers. Let $0<a_1\leq a_2\leq \cdots \leq a_{k+1}<1$
 be $k+1$ real numbers.
Define $I_1$ and $I_2$ as in \eqref{5.17}. Then
\begin{equation*}
\begin{aligned}
I_1 &= a_{k+1}\int_{S^{k-1}}du_1\int_{0}^1dt\int_{0}^1
\rho_{\bar{G}}^k(u_1)\left(\rho_{\bar{G}}^2(u_1)+a_{k+1}^2t^2
+s^2\right)^{\frac{q-n}{2}}s^{n-k-2}ds,\\
I_2 &= a_{k+1}\int_{S^{k-1}}du_1\int_{0}^1dt\int_1^\infty
\rho_{\bar{G}}^k(u_1)\left(\rho_{\bar{G}}^2(u_1)+a_{k+1}^2t^2+
s^2\right)^{\frac{q-n}{2}}s^{n-k-q-2}ds.
\end{aligned}
\end{equation*}
\end{lemma}

\begin{proof}
Fix $u_1\in S^{k-1}$ and $\theta \in \left[0,\arctan \frac{a_{k+1}}{\rho_{\bar{G}}(u_1)}\right]$.
We first make the change of variable
\begin{equation*}
s=\frac{\rho_{\bar{G}}(u_1)}{\cos \theta}\tan \phi.
\end{equation*}
By direct computation,
\begin{equation*}
\frac{d \phi}{d s} = \frac{\cos^2\phi \cos\theta}{\rho_{\bar G}(u_1)}=
\frac{\rho_{\bar{G}}(u_1)\cos \theta}{\rho_{\bar{G}}^2(u_1)+s^2\cos^2\theta},
\end{equation*}
and
\begin{equation*}
\begin{aligned}
\cos\phi & = \frac{1}{\sqrt{1+\frac{\cos^2\theta}{\rho_{\bar{G}}^2(u_1)}s^2}}
 = \rho_{\bar{G}}(u_1)\left(\rho_{\bar{G}}^2(u_1)+s^2\cos^2\theta\right)^{-\frac{1}{2}},\\
\sin\phi &= \frac{1}{\sqrt{1+\frac{\cos^2\theta}{\rho_{\bar{G}}^2(u_1)}s^2}}
\frac{s\cos\theta}{\rho_{\bar{G}}(u_1)}
=  s\cos\theta \left(\rho_{\bar{G}}^2(u_1)+s^2\cos^2\theta\right)^{-\frac{1}{2}}.
\end{aligned}
\end{equation*}
Hence,
\begin{equation*}
\begin{aligned}
\rho_{\bar{G}}^q&(u_1)\cos^{k-1-q}\theta\cos^{k-q}\phi \sin^{n-k-2}\phi\, d\phi \\
&= \rho_{\bar{G}}^q(u_1)\cos^{k-1-q}\theta \rho_{\bar{G}}^{k-q}(u_1)
\left(\rho_{\bar{G}}^2(u_1)+s^2\cos^2\theta\right)^{-\frac{k-q}{2}}
 s^{n-k-2}\cos^{n-k-2}\theta \\
&\hskip 128pt
\left(\rho_{\bar{G}}^2(u_1)+s^2\cos^2\theta\right)^{-\frac{n-k-2}{2}}
 \frac{\rho_{\bar{G}}(u_1)\cos \theta}{\rho_{\bar{G}}^2(u_1)+s^2\cos^2\theta}ds\\
&= \rho_{\bar{G}}^{k+1}(u_1)\cos^{n-q-2}\theta
\left(\rho_{\bar{G}}^2(u_1)+s^2\cos^2\theta\right)^{-\frac{n-q}{2}}s^{n-k-2}ds,
\end{aligned}
\end{equation*}
and
\begin{equation*}
\begin{aligned}
\cos&^{k-1}\theta \cos^{k}\phi \sin^{n-k-q-2}\phi\, d\phi\\
&= \cos^{k-1}\theta\rho_{\bar{G}}^k(u_1)
\left(\rho_{\bar{G}}^2(u_1)+s^2\cos^2\theta\right)^{-\frac{k}{2}}
s^{n-k-q-2}\cos^{n-k-q-2}\theta \\
&\hskip 128pt \left(\rho_{\bar{G}}^2(u_1)+s^2\cos^2\theta\right)^{-\frac{n-k-q-2}{2}}
\frac{\rho_{\bar{G}}(u_1)\cos \theta}{\rho_{\bar{G}}^2(u_1)+s^2\cos^2\theta}ds\\
&= \rho_{\bar{G}}^{k+1}(u_1)\cos^{n-q-2}\theta
\left(\rho_{\bar{G}}^2(u_1)+s^2\cos^2\theta\right)^{-\frac{n-q}{2}}s^{n-k-q-2}ds.
\end{aligned}
\end{equation*}
Thus, by \eqref{5.17}, we have
\begin{equation}\label{5.19}
\begin{aligned}
I_1 &= \int_{S^{k-1}}du_1\int_{0}^{\arctan \frac{a_{k+1}}{\rho_{\bar{G}}(u_1)}}d\theta
\int_{0}^{1}\rho_{\bar{G}}^{k+1}(u_1)\cos^{n-q-2}\theta
\left(\rho_{\bar{G}}^2(u_1)+s^2\cos^2\theta\right)^{-\frac{n-q}{2}}s^{n-k-2}ds\\
&= \int_{S^{k-1}}du_1 \int_{0}^{1}ds\int_{0}^{\arctan \frac{a_{k+1}}{\rho_{\bar{G}}(u_1)}}
\rho_{\bar{G}}^{k+1}(u_1)\cos^{n-q-2}\theta\left(\rho_{\bar{G}}^2(u_1)
+s^2\cos^2\theta\right)^{-\frac{n-q}{2}}s^{n-k-2}d\theta,
\end{aligned}
\end{equation}
and
\begin{equation}\label{5.20}
\begin{aligned}
I_2 &= \int_{S^{k-1}}du_1\int_{0}^{\arctan \frac{a_{k+1}}{\rho_{\bar{G}}(u_1)}}d\theta
\int_{1}^{\infty}\rho_{\bar{G}}^{k+1}(u_1)\cos^{n-q-2}\theta
\left(\rho_{\bar{G}}^2(u_1)+s^2\cos^2\theta\right)^{-\frac{n-q}{2}}s^{n-k-q-2}ds\\
&= \int_{S^{k-1}}du_1\int_{1}^{\infty}ds\int_{0}^{\arctan \frac{a_{k+1}}
{\rho_{\bar{G}}(u_1)}}\rho_{\bar{G}}^{k+1}(u_1)\cos^{n-q-2}\theta
\left(\rho_{\bar{G}}^2(u_1)+s^2\cos^2\theta\right)^{-\frac{n-q}{2}}s^{n-k-q-2}d\theta.
\end{aligned}
\end{equation}

Let us now fix $u_1$, $s$ and make the change of variable
\begin{equation*}
t=\frac{\rho_{\bar{G}}(u_1)}{a_{k+1}}\tan \theta.
\end{equation*}
By direct computation,
\begin{equation*}
\frac{d \theta}{d t} = \frac{1}{1+\frac{a_{k+1}^2}
{\rho_{\bar{G}}^2(u_1)}t^2}\frac{a_{k+1}}{\rho_{\bar{G}}(u_1)}
= \frac{\rho_{\bar{G}}(u_1)a_{k+1}}{\rho_{\bar{G}}^2(u_1)+a_{k+1}^2t^2},
\end{equation*}
and
\begin{equation*}
\cos \theta = \frac{1}{\sqrt{1+\frac{a_{k+1}^2}{\rho_{\bar{G}}^2(u_1)}t^2}}
= \rho_{\bar{G}}(u_1)\left(\rho_{\bar{G}}^2(u_1)+a_{k+1}^2t^2\right)^{-\frac{1}{2}}.
\end{equation*}
Hence,
\begin{equation*}
\begin{aligned}
\rho_{\bar{G}}^{k+1}&(u_1)\cos^{n-q-2}\theta\left(\rho_{\bar{G}}^2(u_1)
+s^2\cos^2\theta\right)^{-\frac{n-q}{2}}s^{n-k-2}d\theta\\
= &\rho_{\bar{G}}^{k+1}(u_1)\rho_{\bar{G}}^{n-q-2}(u_1)
\left(\rho_{\bar{G}}^2(u_1)+a_{k+1}^2t^2\right)^{-\frac{n-q-2}{2}}\\
&\hskip 60pt \left(\rho_{\bar{G}}^2(u_1)+s^2\rho_{\bar{G}}^2(u_1)
\left(\rho_{\bar{G}}^2+a_{k+1}^2t^2\right)^{-1}\right)^{-\frac{n-q}{2}}s^{n-k-2}
\frac{\rho_{\bar{G}}(u_1)a_{k+1}}{\rho_{\bar{G}}^2(u_1)+a_{k+1}^2t^2}dt\\
= &a_{k+1}\rho_{\bar{G}}^k(u_1)\left(\rho_{\bar{G}}^2(u_1)
+a_{k+1}^2t^2+s^2\right)^{-\frac{n-q}{2}}s^{n-k-2}dt,
\end{aligned}
\end{equation*}
and similarly,
\begin{equation*}
\begin{aligned}
\rho_{\bar{G}}^{k+1}&(u_1)\cos^{n-q-2}\theta\left(\rho_{\bar{G}}^2(u_1)
+s^2\cos^2\theta\right)^{-\frac{n-q}{2}}s^{n-k-q-2}d\theta\\
=& a_{k+1}\rho_{\bar{G}}^k(u_1)\left(\rho_{\bar{G}}^2(u_1)
+a_{k+1}^2t^2+s^2\right)^{-\frac{n-q}{2}}s^{n-k-q-2}dt.
\end{aligned}
\end{equation*}
Thus, by \eqref{5.19} and \eqref{5.20}, we have
\begin{equation*}
\begin{aligned}
I_1 &= \int_{S^{k-1}}du_1 \int_{0}^{1}ds\int_{0}^{1}a_{k+1}
\rho_{\bar{G}}^k(u_1)\left(\rho_{\bar{G}}^2(u_1)+a_{k+1}^2t^2
+s^2\right)^{-\frac{n-q}{2}}s^{n-k-2}dt\\
	&= \int_{S^{k-1}}du_1 \int_{0}^{1}dt\int_{0}^{1}a_{k+1}
\rho_{\bar{G}}^k(u_1)\left(\rho_{\bar{G}}^2(u_1)+a_{k+1}^2t^2
+s^2\right)^{-\frac{n-q}{2}}s^{n-k-2}ds,
\end{aligned}
\end{equation*}
and
\begin{equation*}
\begin{aligned}
I_2 &= \int_{S^{k-1}}du_1 \int_{1}^{\infty}ds\int_{0}^{1}a_{k+1}
\rho_{\bar{G}}^k(u_1)\left(\rho_{\bar{G}}^2(u_1)+a_{k+1}^2t^2
+s^2\right)^{-\frac{n-q}{2}}s^{n-k-q-2}dt\\
	&= \int_{S^{k-1}}du_1 \int_{0}^{1}dt\int_{1}^{\infty}a_{k+1}
\rho_{\bar{G}}^k(u_1)\left(\rho_{\bar{G}}^2(u_1)+a_{k+1}^2t^2
+s^2\right)^{-\frac{n-q}{2}}s^{n-k-q-2}ds.
\end{aligned}
\end{equation*}
\end{proof}

\begin{lemma}\label{5.21}
Let $n\geq 3$ and $1\leq k\leq n-2$ be integers. Let $0<a_1\leq a_2\leq \cdots \leq a_{k+1}<1$
be $k+1$ real numbers.
Define $I_3$ and $I_4$ as in \eqref{5.17}. Then
\begin{equation*}
\begin{aligned}
I_3 &= a_{k+1}\int_{S^{k-1}}du_1 \int_1^\infty dt\int_{0}^1
\rho_{\bar{G}}^k(u_1)\left(\rho_{\bar{G}}^2(u_1)+a_{k+1}^2t^2
+s^2t^2\right)^{\frac{q-n}{2}}t^{n-k-q-1}s^{n-k-2}ds,\\
I_4 &= a_{k+1}\int_{S^{k-1}}du_1\int_{1}^\infty dt \int_1^\infty
\rho_{\bar{G}}^k(u_1)\left(\rho_{\bar{G}}^2(u_1)+a_{k+1}^2t^2
+s^2t^2\right)^{\frac{q-n}{2}}t^{n-k-q-1}s^{n-k-q-2}ds.
\end{aligned}
\end{equation*}
\end{lemma}

\begin{proof}
The proof is similar to that of Lemma \ref{5.18}. We give only the main steps.
We first fix $u_1\in S^{k-1}$,
$\theta \in \left(\arctan\frac{a_{k+1}}{\rho_{\bar{G}}(u_1)},\pi/2\right]$
and make the change of variable
\begin{equation*}
s = \frac{a_{k+1}}{\sin\theta}\tan\phi.
\end{equation*}
By direct computation,
\begin{equation*}
\frac{d \phi}{d s}=\frac{a_{k+1}\sin\theta}{a_{k+1}^2+s^2\sin^2\theta},
\end{equation*}
and
\begin{equation*}
\cos\phi =a_{k+1}\left(a_{k+1}^2+s^2\sin^2\theta\right)^{-\frac{1}{2}}, \quad
\sin\phi =s\sin\theta\left(a_{k+1}^2+s^2\sin^2\theta\right)^{-\frac{1}{2}}.
\end{equation*}
Hence,
\begin{equation*}
\begin{aligned}
a_{k+1}^q&\cos^{k-1}\theta\sin^{-q}\theta\cos^{k-q}\phi\sin^{n-k-2}\phi\, d\phi\\
=&\, a_{k+1}^{k+1}\cos^{k-1}\theta\sin^{n-k-q-1}\theta
\left(a_{k+1}^2+s^2\sin^2\theta\right)^{-\frac{n-q}{2}}s^{n-k-2}ds,
\end{aligned}
\end{equation*}
and
\begin{equation*}
\begin{aligned}
\cos^{k-1}&\theta\cos^{k}\phi\sin^{n-k-q-2}\phi\, d\phi\\
=&\, a_{k+1}^{k+1}\cos^{k-1}\theta\sin^{n-k-q-1}\theta
\left(a_{k+1}^2+s^2\sin^2\theta\right)^{-\frac{n-q}{2}}s^{n-k-q-2}ds.
\end{aligned}
\end{equation*}
Thus, by \eqref{5.17},
\begin{equation}
\label{5.22}
I_3 = \int_{S^{k-1}}du_1\int_{0}^{1}ds\int_{\arctan\frac{a_{k+1}}
{\rho_{\bar{G}}(u_1)}}^{\pi/2}a_{k+1}^{k+1}\cos^{k-1}\theta\sin^{n-k-q-1}\theta
\left(a_{k+1}^2+s^2\sin^2\theta\right)^{-\frac{n-q}{2}}s^{n-k-2}d\theta,
\end{equation}
and
\begin{equation}\label{5.23}
I_4 = \int_{S^{k-1}}du_1\int_{1}^{\infty}ds\int_{\arctan\frac{a_{k+1}}
{\rho_{\bar{G}}(u_1)}}^{\pi/2}a_{k+1}^{k+1}\cos^{k-1}\theta\sin^{n-k-q-1}\theta
\left(a_{k+1}^2+s^2\sin^2\theta\right)^{-\frac{n-q}{2}}s^{n-k-q-2}d\theta.
\end{equation}

Let us now fix $u_1$, $s$ and make the change of variable
\begin{equation*}
t=\frac{\rho_{\bar{G}}(u_1)}{a_{k+1}}\tan\theta.
\end{equation*}
By direct computation,
\begin{equation*}
\frac{d \theta}{d t}
 = \frac{\rho_{\bar{G}}(u_1)a_{k+1}}{\rho_{\bar{G}}^2(u_1)+a_{k+1}^2t^2},
\end{equation*}
and
\begin{equation*}
\cos\theta = \rho_{\bar{G}}(u_1)\left(\rho_{\bar{G}}^2(u_1)+a_{k+1}^2t^2\right)^{-\frac{1}{2}},
\quad \sin\theta =a_{k+1}t\left(\rho_{\bar{G}}^2(u_1)+a_{k+1}^2t^2\right)^{-\frac{1}{2}}.
\end{equation*}
Hence,
\begin{equation*}
\begin{aligned}
a_{k+1}^{k+1}&\cos^{k-1}\theta\sin^{n-k-q-1}\theta
\left(a_{k+1}^2+s^2\sin^2\theta\right)^{-\frac{n-q}{2}}s^{n-k-2}d\theta\\
=&\, a_{k+1}\rho_{\bar{G}}^k(u_1)\left(\rho_{\bar{G}}^2(u_1)
+a_{k+1}^2t^2+s^2t^2\right)^{-\frac{n-q}{2}}t^{n-k-q-1}s^{n-k-2}dt,
\end{aligned}
\end{equation*}
and
\begin{equation*}
\begin{aligned}
a_{k+1}^{k+1}&\cos^{k-1}\theta\sin^{n-k-q-1}\theta
\left(a_{k+1}^2+s^2\sin^2\theta\right)^{-\frac{n-q}{2}}s^{n-k-q-2}d\theta\\
=&\, a_{k+1}\rho_{\bar{G}}^k(u_1)\left(\rho_{\bar{G}}^2(u_1)
+a_{k+1}^2t^2+s^2t^2\right)^{-\frac{n-q}{2}}t^{n-k-q-1}s^{n-k-q-2}dt.
\end{aligned}
\end{equation*}
Thus, by \eqref{5.22} and \eqref{5.23},
\begin{equation*}
I_3 =\int_{S^{k-1}}du_1\int_{1}^{\infty}dt\int_{0}^{1}a_{k+1}
\rho_{\bar{G}}^k(u_1)\left(\rho_{\bar{G}}^2(u_1)+a_{k+1}^2t^2
+s^2t^2\right)^{-\frac{n-q}{2}}t^{n-k-q-1}s^{n-k-2}ds,
\end{equation*}
and
\begin{equation*}
I_4 = \int_{S^{k-1}}du_1\int_{1}^{\infty}dt\int_{1}^{\infty}a_{k+1}
\rho_{\bar{G}}^k(u_1)\left(\rho_{\bar{G}}^2(u_1)+a_{k+1}^2t^2
+s^2t^2\right)^{-\frac{n-q}{2}}t^{n-k-q-1}s^{n-k-q-2}ds.
\end{equation*}
\end{proof}

The next lemma provides upper bounds for the integrals $I_1,I_2,I_3$, and $I_4$.

\begin{lemma}\label{5.24}
Let $n\geq 3$ and $1\leq k\leq n-2$ be integers,
$0<a_1\leq a_2\leq \cdots \leq a_{k+1}<1$ be $k+1$ real numbers,
and $q$ be a real number satisfying $k<q<k+1$.
 Define $I_1, \cdots, I_4$ as in \eqref{5.17}.
Then there exist constants $c_i(k,q,n)$ such that
\begin{equation}
	\label{5.26} I_i\leq c_i(k,q,n)a_1\cdots a_k a_{k+1}^{q-k}, \quad
i=1,2,3,4.
\end{equation}
\end{lemma}

\begin{proof}
By the volume formula \eqref{eq volume} applied in $\mathbb{R}^k$,
the fact the $\bar{G}$ is an ellipsoid in $\mathbb{R}^k$,
and the volume formula for an ellipsoid,
\begin{equation}\label{5.25}
\int_{S^{k-1}}\rho_{\bar{G}}^k(u_1)\,du_1=k\omega_k a_1\cdots a_k.
\end{equation}

By Lemma \ref{5.18}, $q<n$, and the facts that
$n-k-1>0$ and $q-k-1\neq 0$, we have
\begin{equation}
	\begin{aligned}
I_1&\leq a_{k+1}\int_{S^{k-1}}du_1\int_{0}^1 dt \int_{0}^1
\rho_{\bar{G}}^k(u_1)\left(a_{k+1}^2t^2+s^2\right)^\frac{q-n}{2}s^{n-k-2}ds\\
&\leq a_{k+1}\int_{S^{k-1}}du_1\int_{0}^1dt \int_{0}^{a_{k+1}t}
\rho_{\bar{G}}^k(u_1)(a_{k+1}t)^{q-n}s^{n-k-2}ds\\
	 &\hskip 130pt +a_{k+1}\int_{S^{k-1}}du_1\int_{0}^1dt\int_{a_{k+1}t}^1
\rho_{\bar{G}}^k(u_1)s^{q-n}s^{n-k-2}ds\\
&= a_{k+1}\int_{S^{k-1}}du_1\int_{0}^1\rho_{\bar{G}}^k(u_1)
\frac{1}{n-k-1}(a_{k+1}t)^{q-k-1}dt\\
	&\hskip 80pt +a_{k+1}\int_{S^{k-1}}du_1\int_{0}^1 \rho_{\bar{G}}^k(u_1)
\frac{1}{q-k-1}\left(1-(a_{k+1}t)^{q-k-1}\right)dt\\
&= \left(\frac{1}{n-k-1}+\frac{1}{k+1-q}\right)a_{k+1}^{q-k}\int_{S^{k-1}}\,du_1
\int_{0}^1 \rho_{\bar{G}}^k(u_1)t^{q-k-1}dt\\
	&\hskip 210pt +\frac{1}{q-k-1}a_{k+1}\int_{S^{k-1}}\rho_{\bar{G}}^k(u_1)\,du_1.
	\end{aligned}
\end{equation}

Note that $-1<q-k-1<0$. Hence
\begin{equation}
	\begin{aligned}
I_1\leq & \left(\frac{1}{n-k-1}+\frac{1}{k+1-q}\right)
a_{k+1}^{q-k}\int_{S^{k-1}}du_1\int_{0}^1 \rho_{\bar{G}}^k(u_1)t^{q-k-1}dt\\
		= & \left(\frac{1}{n-k-1}+\frac{1}{k+1-q}\right)\frac{1}{q-k}a_{k+1}^{q-k}
\int_{S^{k-1}}\rho_{\bar{G}}^k(u_1)\,du_1\\
		=& \left(\frac{1}{n-k-1}+\frac{1}{k+1-q}\right)
\frac{k\omega_k}{q-k} a_1\cdots a_ka_{k+1}^{q-k},
	\end{aligned}
\end{equation}
where the last equality follows from \eqref{5.25}. This establishes \eqref{5.26} for $I_1$.

Since $q<n$, we have, by Lemma \ref{5.18},
\begin{equation}
	\begin{aligned}
		I_2&\leq a_{k+1}\int_{S^{{k-1}}}du_1\int_{0}^1dt\int_{1}^\infty
\rho_{\bar{G}}^k(u_1)s^{q-n}s^{n-k-q-2}ds\\
		&=\frac{1}{k+1}a_{k+1}\int_{S^{k-1}}\rho_{\bar{G}}^{k}(u_1)\,du_1\\
		&=\frac{k\omega_k}{k+1}a_1\cdots a_k a_{k+1},
	\end{aligned}
\end{equation}
where the last equality follows from \eqref{5.25}. By the fact that $q-k<1$,
together with the fact that $0<a_{k+1}<1$, we have
\begin{equation}
	I_2\leq \frac{k\omega_k}{k+1}a_1\cdots a_k a_{k+1}^{q-k}.
\end{equation}
This establishes \eqref{5.26} for $I_2$.

By Lemma \ref{5.21}, $q<n$, and the fact that $n-k-1>0$, $k>0$, and $q-k-1<0$, we have
\begin{equation}
	\begin{aligned}
		I_3 \leq & a_{k+1}\int_{S^{k-1}}du_1\int_{1}^\infty dt \int_{0}^1
\rho_{\bar{G}}^k(u_1)\left(a_{k+1}^2t^2+s^2t^2\right)^\frac{q-n}{2}t^{n-k-q-1}s^{n-k-2}ds\\
		\leq & a_{k+1}\int_{S^{k-1}}du_1 \int_{1}^\infty dt\int_0^{a_{k+1}}
\rho_{\bar{G}}^k(u_1)(a_{k+1}t)^{q-n}t^{n-k-q-1}s^{n-k-2}ds\\
		 &\hskip 120pt +a_{k+1}\int_{S^{k-1}}du_1\int_{1}^{\infty}dt\int_{a_{k+1}}^1
\rho_{\bar{G}}^k(u_1)(st)^{q-n}t^{n-k-q-1}s^{n-k-2}ds\\
		=& \frac{1}{k(n-k-1)}a_{k+1}^{q-k}\int_{S^{k-1}} \rho_{\bar{G}}^{k}(u_1)\,du_1
		 +\frac{1}{k(q-k-1)}a_{k+1}\left(1-a_{k+1}^{q-k-1}\right)
\int_{S^{k-1}}\rho_{\bar{G}}^k(u_1)\,du_1\\
		=& \frac{1}{k}\left(\frac{1}{n-k-1}+\frac{1}{k+1-q}\right)a_{k+1}^{q-k}
\int_{S^{n-1}}\rho_{\bar{G}}^k(u_1)\,du_1
		+\frac{1}{k(q-k-1)}a_{k+1}\int_{S^{k-1}}\rho_{\bar{G}}^k(u_1)\,du_1.
	\end{aligned}
\end{equation}
Note that $q-k-1<0$. Hence,
\begin{equation}
	\begin{aligned}
		I_3&\leq \frac{1}{k}\left(\frac{1}{n-k-1}+\frac{1}{k+1-q}\right)a_{k+1}^{q-k}
\int_{S^{n-1}}\rho_{\bar{G}}^k(u_1)\,du_1\\
		&= \left(\frac{1}{n-k-1}+\frac{1}{k+1-q}\right)\omega_k a_1\cdots a_k a_{k+1}^{q-k},
	\end{aligned}
\end{equation}
where the last equality follows from \eqref{5.25}. This establishes \eqref{5.26} for $I_3$.

Since $q<n$, we have, by Lemma \ref{5.21},
\begin{equation}
	\begin{aligned}
		I_4 &\leq a_{k+1}\int_{S^{k-1}}du_1\int_{1}^\infty dt \int_{1}^\infty
\rho_{\bar{G}}^k(u_1)(st)^{q-n}t^{n-k-q-1}s^{n-k-q-2}ds\\
		&= \frac{1}{k(k+1)}a_{k+1}\int_{S^{k-1}}\rho_{\bar{G}}^k(u_1)\,du_1\\
		&=\frac{\omega_k}{k+1}a_1\cdots a_k a_{k+1}\\
		&\leq \frac{\omega_k}{k+1}a_1\cdots a_k a_{k+1}^{q-k},
	\end{aligned}
\end{equation}
where the second to the last equality follows from \eqref{5.25} and the last inequality
follows from $q-k<1$ and $a_{k+1}<1$. This establishes \eqref{5.26} for $I_4$.
\end{proof}

Lemmas \ref{5.16} and \ref{5.24} imply the following lemma.

\begin{lemma}\label{5.27} Let $e_1, \cdots, e_n$
be an orthonormal basis in $\mathbb{R}^n$, $n\geq 3$. Suppose $1\leq k\leq n-2$ is
an integer and $0<a_1\leq a_2\leq \cdots \leq a_{k+1}<1$. Let $G$ be the convex body
defined by
\begin{equation*}
	G=\Big\{x\in \mathbb{R}^n: \frac{|x\cdot e_1|^2}{a_1^2}+\cdots
+\frac{|x\cdot e_k|^2}{a_k^2}\leq 1,\ |x\cdot e_{k+1}|\leq a_{k+1},\
|x\cdot e_{k+2}|^2+\cdots +|x\cdot e_n|^2\leq 1\Big\}.
\end{equation*}
If $q$ is a real number satisfying $k<q<k+1$, then
\begin{equation*}
	\wt W_{n-q}(G)\leq c(k,q,n)a_1\cdots a_k a_{k+1}^{q-k}.
\end{equation*}
\end{lemma}

\section{Solutions to the maximization problems and the dual Minkowski problem}\label{6.0}

In this section we solve Maximization Problem II and thus, by Lemmas \ref{3.9} and \ref{3.12}, solving the even dual Minkowski problem for $1 < q < n$.
The solution for  $n-1\leq q<n$ relies
on Lemmas \ref{4.9} and \ref{5.2}, while the solution for $1<q<n-1$
uses Lemmas \ref{4.9} and \ref{5.27}.

\begin{lemma}\label{5.3}
Let $\mu$ be a non-zero even finite Borel measure on $S^{n-1}$ and $1<q<n$.
If $\mu$ satisfies the $q$-th subspace mass inequality, then there exists
$K'\in \mathcal{K}_e^n$ such that
\begin{equation*}\label{max}
\Phi_\mu(K')=\sup\{\Phi_\mu(K): K\in \mathcal{K}_e^n\}.
\end{equation*}
\end{lemma}

\begin{proof}
Let $\{K_l\}$ be a maximizing sequence; i.e., $K_l\in \mathcal{K}_e^n$ and
\begin{equation*}
	\lim_{l\rightarrow \infty}\Phi_\mu(K_l)=\sup\{\Phi_\mu(K): K\in \mathcal{K}_e^n\}.
\end{equation*}
Since $\Phi_\mu$ is scale invariant, we may assume that $K_l$ has diameter $1$.
By Blaschke's selection theorem, there is a subsequence that
converges to an origin-symmetric compact convex set $K_0$.
By the continuity of $\Phi_\mu$ with respect to the Hausdorff metric,
if $K_0$ has nonempty interior, then $K' = K_0$ satisfies \eqref{max}, proving the theorem and solving Maximum Problem II.
To prove that $K_0$ has nonempty interior, we argue by contradiction and assume that $K_0$ is contained in some proper subspace of $\mathbb{R}^n$.

For each $K_l$, we choose an orthonormal basis $e_{1l},\cdots, e_{nl}$
and real numbers $0<a_{1l}\leq a_{2l}\leq \cdots\leq a_{nl}< 1$ such that the ellipsoid
\begin{equation*}
Q_l=\left\{x\in \mathbb{R}^n:\frac{|x\cdot e_{1l}|^2}{a_{1l}^2}+\cdots
+\frac{|x\cdot e_{nl}|^2}{a_{nl}^2}\leq 1\right\}
\end{equation*}
satisfies
\begin{equation}
	\label{6.1} Q_l\subset K_l\subset\sqrt{n} Q_l.
\end{equation}
In particular, we choose $Q_l$ to be the John ellipsoid associated with $K_l$.
Since the diameter of $K_l$ is $1$,
the diameter of $\sqrt{n} Q_l$ is greater than 1, and therefore
$a_{nl}\geq \frac{1}{2\sqrt{n}}$. By taking subsequences, we may assume that
the sequence of orthonormal bases $\{e_{1l},\ldots, e_{nl}\}$
and the sequences $\{a_{1l}\},\ldots, \{a_{nl}\}$
converge. Since $K_0$ is contained in some proper subspace
of $\mathbb{R}^n$, there must exists $1\leq k\leq n-1$ such that
$a_{1l},\cdots, a_{kl}\rightarrow 0$ as $l\rightarrow \infty$ and
$\lim_{l\rightarrow\infty}a_{il}=a_i$ for $k<i\leq n$ and some $a_{k+1},\cdots, a_n>0$.

We first consider the case of $n-1\leq q<n$.
Lemma \ref{4.9} and \eqref{6.1} allow us to conclude that there exist $t_0,\delta_0,l_0>0$
such that for all $l>l_0$,
\begin{equation}\label{6.2}
E_\mu(K_l)\leq E_\mu(Q_l)\leq -\frac{1}{q}\log(a_{1l}\cdots a_{n-1,l})
+t_0\log a_{1l}+c(n,q,t_0,\delta_0)
\end{equation}
Define the ellipsoidal cylinder,
\begin{equation*}
T_l=\left\{x\in\mathbb{R}^n: \frac{|x\cdot e_{1l}|^2}{a_{1l}^2}+\cdots
+\frac{|x\cdot e_{n-1,l}|^2}{a_{n-1,l}^2}\leq 1 \text{ and }|x\cdot e_{nl}|\leq 1\right\}.
\end{equation*}
Since $a_{nl}\leq 1$, we have
\begin{equation}
	\label{6.3} K_l\subset \sqrt{n} Q_l\subset \sqrt{n}T_l.
\end{equation}
Since $t_0>0$, one can
choose $q_0$ so that $q<q_0<n$ and $(n-1)\left(\frac{1}{q_0}-\frac{1}{q}\right)+t_0>0$.
By \eqref{2.7}, the monotonicity of $L_p$ norms with the fact that $q_0>q$, \eqref{6.3},
the homogeneity of a dual quermassintegral, and Lemma \ref{5.2}, we have
\begin{equation}\label{6.4}
\begin{aligned}
\frac{1}{q}\log \widetilde{W}_{n-q}(K_l)
&= \log\left(\frac{1}{n\omega_n}\int_{S^{n-1}}\rho_{K_l}^q(u)\,du\right)^\frac{1}{q}
   +\frac1q \log \omega_n\\
&\leq \log\left(\frac{1}{n\omega_n}\int_{S^{n-1}}\rho_{K_l}^{q_0}(u)\,du\right)^\frac{1}{q_0}
   +\frac1q \log \omega_n\\
& \leq \log\left(\frac{1}{n\omega_n}\int_{S^{n-1}}
\rho_{\sqrt{n}T_l}^{q_0}(u)\,du\right)^\frac{1}{q_0}
   +\frac1q \log \omega_n\\
& = \frac{1}{q_0}\log \widetilde{W}_{n-q_0}(T_l)+c(n,q,q_0)\\
		&\leq \frac{1}{q_0}\log(a_{1,l}\cdots a_{n-1,l})+c(n,q,q_0).
\end{aligned}
\end{equation}

Equations \eqref{6.2}, \eqref{6.4}, $q_0>q$, and $a_{1l}\leq \cdots \leq a_{n-1,l}$
now imply that
\begin{equation*}
	\begin{aligned}
\Phi_\mu(K_l) &=E_\mu(K_l) + \frac1q \log \wt W_{n-q}(K_l) \\
&\leq \left(\frac{1}{q_0}-\frac{1}{q}\right)\log(a_{1l}\cdots a_{n-1,l})
+t_0\log a_{1l}+c(n,q,q_0,t_0,\delta_0)\\
&\leq \left((n-1)\left(\frac{1}{q_0}-\frac{1}{q}\right)
+t_0\right)\log a_{1l}+c(n,q,q_0,t_0,\delta_0)\\
		&\rightarrow -\infty,
	\end{aligned}
\end{equation*}
as $l\rightarrow \infty$. Here, the last step follows because
 $q_0$ was chosen so that $(n-1)\left(\frac{1}{q_0}-\frac{1}{q}\right)+t_0>0$
and that $a_{1l}\rightarrow 0$. This contradicts the assumption that
$\{K_l\}$ is a maximizing sequence.

Next, we consider the case when $1<q<n-1$. If $n = 2$, then $n-1 = 1$ and therefore the proof above applies. We therefore assume that $n > 3$,

By \eqref{6.1} and Lemma \ref{4.9}, there exists $t_0,\delta_0, l_0>0$
such that for each $l>l_0$, we have
\begin{equation}\label{6.5}
E_\mu(K_l)\leq E_\mu(Q_l)\leq -\frac{1}{q}\log(a_{1l}\cdots a_{\lfloor q \rfloor,l})
-\frac{q-\lfloor q \rfloor}{q}\log a_{\lfloor q \rfloor+1,l}+t_0\log a_{1l}+c(n,t_0,\delta_0).
\end{equation}
Since $t_0>0$, there exists $q_0\in (q,n-1)$ sufficiently close to $q$ so that $q_0$ is a non-integer satisfying $\lfloor q_0 \rfloor=\lfloor q \rfloor$
and $(n-2)\left(\frac{1}{q_0}-\frac{1}{q}\right)+t_0>0$.
Let $k_0$ be the integer so that $q_0-1<k_0<q_0$, that is,
\begin{equation}
	\label{6.6} k_0=\lfloor q_0 \rfloor=\lfloor q\rfloor.
\end{equation}
Let $G_l$ be the Cartesian product of an ellipsoid, a line segment, and a ball given by
\begin{equation*}
	\begin{aligned}
G_l=\Big\{x\in\mathbb{R}^n:\frac{|x\cdot e_{1l}|^2}{a_{1l}^2}+\cdots
+\frac{|x\cdot e_{k_0l}|^2}{a_{k_0l}^2}\leq 1,\ &|x\cdot e_{k_0+1,l}|\leq a_{k_0+1,l},\\
		&|x\cdot e_{k_0+2,l}|^2+\cdots+|x\cdot e_{nl}|^2\leq 1\Big\}.
	\end{aligned}
\end{equation*}
Note that since $a_{1l}\leq \cdots \leq a_{nl}< 1$,
we have $Q_l\subset G_l$. By \eqref{6.1},
\begin{equation}
	\label{6.7}
K_l\subset\sqrt{n} Q_l\subset \sqrt{n}G_l.
\end{equation}
Note that $1\leq k_0\leq n-2$. By \eqref{2.7},
the monotonicity of $L_p$ norms with the fact that $q_0>q$,
\eqref{6.7}, the homogeneity of dual quermassintegral, Lemma \ref{5.27}, and \eqref{6.6},
\begin{equation}
	\label{6.8}
	\begin{aligned}
\frac1q \log \wt W_{n-q}(K_l)& = \log\left(\frac{1}{n\omega_n}\int_{S^{n-1}}
\rho_{K_l}^q(u)\,du\right)^\frac{1}{q}+ \frac1q \log\omega_n \\
&\leq \log \left(\frac{1}{n\omega_n}\int_{S^{n-1}}\rho_{K_l}^{q_0}(u)\,du\right)^\frac{1}{q_0}
+\frac1q \log\omega_n\\
&\leq \log \left(\frac{1}{n\omega_n}\int_{S^{n-1}}
\rho_{\sqrt{n}G_l}^{q_0}(u)\,du\right)^\frac{1}{q_0}+\frac1q \log\omega_n\\
& = \frac{1}{q_0}\log \widetilde{W}_{n-q_0}(G_l)+c(n,q,q_0)\\
&\leq \frac{1}{q_0}\log (a_{1l}\cdots a_{k_0l})
+\frac{q_0-k_0}{q_0}\log a_{k_0+1,l}+c(n,k_0,q,q_0)\\
&= \frac{1}{q_0}\log (a_{1l}\cdots a_{\lfloor q\rfloor,l})
+\frac{q_0-\lfloor q\rfloor}{q_0}\log a_{\lfloor q\rfloor+1,l}+c(n,q,q_0).
	\end{aligned}
\end{equation}
By \eqref{6.5}, \eqref{6.8}, the fact that $q<q_0<n-1$, and the fact that
$0<a_{1l}\leq \cdots\leq a_{nl}< 1$, we conclude that when $l>l_0$,
\begin{equation*}
	\begin{aligned}
\Phi_\mu(K_l)&=E_\mu(K_l) + \frac1q \log \wt W_{n-q}(K_l) \\
&\leq \left(\frac{1}{q_0}-\frac{1}{q}\right)\log(a_{1l}\cdots
a_{\lfloor q\rfloor,l})+\lfloor q\rfloor\left(\frac{1}{q}-\frac{1}{q_0}\right)
\log a_{\lfloor q\rfloor+1,l}+t_0\log a_{1l}+c(n,\delta_0,t_0,q,q_0)\\
&\leq \left(\frac{1}{q_0}-\frac{1}{q}\right)\log(a_{1l}\cdots
a_{\lfloor q\rfloor,l})+t_0\log a_{1l}+c(n,\delta_0,t_0,q,q_0)\\
&\leq \lfloor q\rfloor\left(\frac{1}{q_0}-\frac{1}{q}\right)
\log a_{1l}+t_0\log a_{1l}+c(n,\delta_0,t_0,q,q_0)\\
&\le \left((n-2)\Big(\frac1{q_0}-\frac1q\Big)+t_0\right) \log a_{1l} +c(n,\delta_0,t_0,q,q_0)\\
&\rightarrow -\infty,
	\end{aligned}
\end{equation*}
as $l\rightarrow \infty$. Here the last step uses the fact that
 $(n-2)\left(\frac{1}{q_0}-\frac{1}{q}\right)+t_0>0$ and
 that $\lim_{l\rightarrow \infty}a_{1l}=0$. This contradicts
the assumption that $\{K_l\}$ is a maximizing sequence, thereby proving the lemma.
\end{proof}

The above lemma combined with Lemma \ref{3.13} gives a complete solution to
the even dual Minkowski problem for $1<q<n$.
\begin{theorem}
\label{theorem 2} If $0< q<n$ and $\mu$ is a non-zero even finite Borel measure
on $\sn$, then
there exists $K\in \mathcal{K}_e^n$ such that $\mu = \widetilde{C}_q(K,\cdot)$
if and only if $\mu$ satisfies the $q$-th subspace mass inequality \eqref{4.1}.
\end{theorem}

Note again that the even dual Minkowski problem, when $0<q\leq 1$,
was solved in \cite{HLYZ}.  The necessary condition of Theorem \ref{theorem 2},
when $1<q<n$, was proved in \cite{BH},
and the sufficient condition of Theorem \ref{theorem 2}, when $q\in \{2, \ldots, n-1\}$
is an integer, was
proved in \cite{zhao}.

\end{document}